\documentclass[11pt]{article}

\usepackage[english]{babel}
\usepackage{amsthm}
\newtheorem{theorem}{Theorem}[section]
\newtheorem{lemma}[theorem]{Lemma}

\usepackage[utf8]{inputenc}
\usepackage{mathtools}

\usepackage{amsmath,amssymb,amstext,amsfonts}
\usepackage{relsize}
\usepackage{algorithm}
\usepackage{lipsum}
\usepackage{algpseudocode} 

\usepackage{enumerate}
\usepackage[all,arc]{xy}
\usepackage{mathrsfs}
\usepackage{tcolorbox}
\usepackage{graphicx}
\usepackage{subcaption}
\usepackage{float}
\usepackage[margin=0.9 in]{geometry}
\usepackage{listings}
\usepackage{xcolor}
\usepackage{subcaption} 

\usepackage{systeme}
\usepackage{comment}
\usepackage{dsfont}
\newtheorem*{remark}{Remark}

\newcommand{\calg}{\mathcal{G}}
\newcommand{\caln}{\mathcal{N}}
\newcommand{\calu}{\mathcal{U}}
\newcommand{\calf}{\mathcal{F}}
\newcommand{\calk}{\mathcal{K}}
\newcommand{\calm}{\mathcal{M}}
\newcommand{\bR}{\mathbb{R}}
\newcommand{\E}{\mathbb{E}}

\newcommand{\be}{\begin{equation}}
\newcommand{\ee}{\end{equation}}

\date{}

\begin{document}
\title{Convergence Analysis for A Stochastic Maximum Principle Based Data Driven Feedback Control Algorithm}
\author{
Siming Liang \thanks{Department of Mathematics, Florida State University, Tallahassee, Florida, USA, \tt{sliang@fsu.edu}}, Hui Sun \thanks{Citi Bank, Delaware}, Richard Archibald \thanks{Devision of Computational Science and Mathematics, Oak Ridge National Laboratory, Oak Ridge, Tennessee, USA}, Feng Bao \thanks{Department of Mathematics, Florida State University, Tallahassee, Florida, USA } }
\maketitle

\begin{abstract}
This paper presents convergence analysis of a novel data-driven feedback control algorithm designed for generating online controls based on partial noisy observational data. The algorithm comprises a particle filter-enabled state estimation component, estimating the controlled system's state via indirect observations, alongside an efficient stochastic maximum principle type optimal control solver. By integrating weak convergence techniques for the particle filter with convergence analysis for the stochastic maximum principle control solver, we derive a weak convergence result for the optimization procedure in search of optimal data-driven feedback control. Numerical experiments are performed to validate the theoretical findings.
\end{abstract}

\section{Introduction}
In this paper, we carry out numerical analysis demonstrating the convergence of a data-driven feedback control algorithm designed for generating online control based on partial noisy observational data. 

Our focus lies in the stochastic feedback control problem, which aims to determine optimal controls. These control actions are used to guide a controlled state dynamical system towards meeting certain optimality conditions, leveraging feedbacks from the system's current state. There are two practical challenges in solving the feedback control problem. First, when the control problem's dimension is high, the computational cost for searching the optimal control escalates exponentially. This is known as the ``curse of dimensionality''.  Second, in numerous scenarios, the state of the controlled system is not directly observable. Therefore, state estimation techniques become necessary to estimate the current state for designing optimal control, with observations gathered to aid in estimating the hidden state.

To address the aforementioned challenges, a novel online data-driven feedback control algorithm has been developed \cite{Bao_CiCP23, archibald2020efficient}. This algorithm introduces a stochastic gradient descent optimal control solver within the stochastic maximum principle framework to combat the high dimensionality issue in optimal control problems. Traditionally, stochastic optimal control problems are solved using dynamical programming or the stochastic maximum principle, both requiring numerical simulations for large differential systems \cite{Bellman1957, Feng_HJB_2013, Peng1990}. However, the stochastic maximum principle stands out for its capability to handle random coefficients in the state model and finite-dimensional terminal state constraints \cite{Yong_control}. In the stochastic maximum principle approach, a system of backward stochastic differential equations (BSDEs) are derived as the adjoint equation of the controlled state process. 
Then, the solution of the adjoint BSDE is utilized to formulate the gradient of the cost function with respect to the control process \cite{GPM_2017, Tang-1998}. However, solving BSDEs numerically entails significant computational costs, especially in high-dimensional problems, demanding a large number of random samples \cite{ZhangJ_BSDE, Zhao_multi}. To bolster efficiency, a sample-wise optimal control solver method has been devised \cite{Bao_EAJAM20}, where the solution of the adjoint BSDE is represented using only one realization or a small batch of samples. This approach justifies the application of stochastic approximation in the optimization procedure  \cite{Convergence-SGLD, Convergence-SGD}, and it shifts the computational cost from solving BSDEs to searching for the optimal control, thereby enhancing overall efficiency \cite{Bao_SNN}.

In data-driven feedback control, optimal filtering methods also play a pivotal role in dynamically estimating the state of the controlled system. Two prominent approaches for \textit{nonlinear} optimal filtering are the Zakai filter and the particle filter. While the Zakai filter aims to compute the conditional probability density function (pdf) for the target dynamical system using a parabolic type stochastic partial differential equation known as the Zakai equation \cite{zakai}. The particle filter, also known as a sequential Monte Carlo method, approximates the desired conditional pdf using the empirical distribution of a set of random samples (particles) \cite{particle-filter}. Although the Zakai filter theoretically offers more accurate approximations for conditional distributions, the particle filter is favored in more practical applications due to the high efficiency of the Monte Carlo method in approximating high-dimensional distributions \cite{MTAC2012}.

The aim of this study is to examine the convergence of the data-driven feedback control algorithm proposed in \cite{archibald2020efficient}, providing mathematical validation for its performance. While convergence in particle filter methods has been well studied \cite{MCMC-PF, crisan2002survey, Kunsch_PF}, this work adopts the analysis technique outlined in \cite{crisan2002survey} to establish weak convergence results for the particle filter regarding the number of particles. Analysis techniques for BSDEs alongside classical convergence results for stochastic gradient descent \cite{Bao_SNN, Bao_first} are crucial for achieving convergence in the stochastic gradient descent optimal control solver. The theoretical framework of this analysis merges the examination of particle filters with the analysis of optimal control, and the overarching objective of this paper is to derive a comprehensive weak convergence result for the optimal data-driven feedback control.

The rest of this paper is organized as follows. In Section 2, we introduce the data driven feedback control algorithm. Convergence analysis will be presented in Section 3, and in Section 4 we conduct multiple numerical experiments to validate our theoretical findings.

\section{An efficient algorithm for data driven feedback control}

We first briefly introduce the data driven feedback control problem that we consider in this work. Then, we shall describe our efficient algorithm for solving the data driven feedback control problem by using a stochastic gradient descent type optimization procedure for the optimal control. 

\subsection{Problem setting for the data driven optimal control problem.}

In probability space $(\Omega, \mathcal{F}, \mathbb{P})$, we consider the following augmented system on time interval $[0,T]$
\begin{equation}\label{XM}
\begin{aligned}
d\begin{pmatrix}X_t\\ M_t\end{pmatrix} =\begin{pmatrix}b(t,X_t,u_t )\\ g(X_t)\end{pmatrix}dt  + \begin{pmatrix}\sigma(t,X_t,u_t)&0\\0&I\end{pmatrix}
d\begin{pmatrix}W_t\\ B_t\end{pmatrix}, \qquad
\begin{pmatrix}X_0 = \xi \\ M_0 = 0 \end{pmatrix},
\end{aligned}
\end{equation}
where $X : = \{X_t\}_{t=0}^T$ is the $\mathbb{R}^d$-dimensional controlled state process with dynamics $b: [0, T] \times \mathbb{R}^d \times \mathbb{R}^m \rightarrow \mathbb{R}^d$, $\sigma: [0, T] \times \mathbb{R}^d \times \mathbb{R}^m \rightarrow \mathbb{R}^{d \times q}$ is the diffusion coefficient for a $d$-dimensional Brownian motion $W$ that perturbs the the state $X$, and $u$ is an $m$-dimensional control process valued in some set $U$ that controls the state process $X$. In the case that the state $X$ is not directly observable, we have an observation process $M$ that collects partial noisy observations on $X$ with observation function $g: \mathbb{R}^d \rightarrow \mathbb{R}^p$, and $B$ is a $p$-dimensional Brownian motion that is independent from $W$. 

Let $\mathbb{F}^B=\{\mathcal{F}_t^B\}_{t \geq 0}$ be the filtration of $B$ augmented by all the $\mathbb{P}$-null sets in $\mathcal{F}$, and $\mathbb{F}^{W,B}\equiv\{\mathcal{F}^{W,B}_t\}_{t\geq 0}$ be the filtration generated by $W$ and $B$ (augmented by $\mathbb{P}$-null sets in $\mathcal{F}$). Under mild conditions, for any square integrable random variable $\xi$ independent of $W$ and $B$, and any $\mathbb{F}^{W,B}$-progressively measurable process $u$ (valued in $U$), Eq. \eqref{XM} admits a unique solution $(X,M)$ which is $\mathbb{F}^{W,B}$-adapted. Next, we let $\mathbb{F}^M=\{\mathcal{F}^M_t\}_{t\geq 0}$ be the filtration generated by $M$ (augmented by all the $\mathbb{P}$-null sets in $\mathcal{F}$). Clearly, $\mathbb{F}^M\subset \mathbb{F}^{W,B}$, and $\mathbb{F}^M \ne \mathbb{F}^W $, $\mathbb{F}^M\ne\mathbb{F}^B$, in general. The $\mathbb{F}^M$ progressively measurable control process, denoted by $u^M$, are control actions driven by the information contained in observational data.

We introduce the set of data driven admissible controls as
$$\mathcal{U}_{ad}[0, T] = \left\{u^{M} : [0, T] \times \Omega  \rightarrow U \subset \mathbb{R}^m \big| u^M \text{ is } \mathbb{F}^M- \text{progressively measurable} \right\},$$
and the cost functional that measures the performance of data driven control $u^M$ is defined as
\begin{equation}\label{cost}
J(u^M) = \E\left[ \int_{0}^{T} f(t, X_t, u_t^M) dt + h(X_T) \right],
\end{equation}
where $f$ is the running cost, and $h$ is the terminal cost.

The goal of the data driven feedback control problem is to find the optimal data driven control $u^{*} \in \mathcal{U}_{ad}[0, T]$ such that
\begin{equation}
J(u^{\ast}) = \inf_{u^M \in \mathcal{U}_{ad}[0, T] } J(u^M).
\end{equation}

\subsection{The algorithm for solving the data driven optimal control problem.}

To solve the data driven feedback control problem, we will use the algorithm from \cite{archibald2020efficient}, which is derived from the stochastic maximum principle.

\subsubsection{The optimization procedure for optimal control}

 When the optimal control $u^{\ast}$ is in the interior of  $\mathcal{U}_{ad}$, the gradient process of the cost functional $J^{\ast}$ with respect to the control process on time interval $t \in [0,T]$ can be derived using the Gâteaux derivative of $u^{\ast}$ and the stochastic maximum principle in the following form:

\begin{equation}
\label{PFSGD:EQ5}
    (J^{\ast})_u^{'}(u_t^{\ast})= E\left[b_u(t,X_t^{\ast},u_t^{\ast})^{\top}Y_t +  \sigma_u(t,X_t^{\ast},u_t^{\ast})^{\top}Z_t + f_u(t,X_t^{\ast},u_t^{\ast})^{\top} | \mathcal{F}_t^M \right],
\end{equation}
where stochastic processes $Y$ and $Z$ are solutions of the following forward backward stochastic differential equations (FBSDEs) system:
\begin{equation}
\label{PFSGD:EQ6}
\begin{cases}
dX_t^{\ast} = b(t,X_t^{\ast},u_t^{\ast})dt + \sigma (t,X_t^{\ast},u_t^{\ast}) dW_t, \hfill X_0 = \xi\\ 
dM_t^{\ast} = \phi(X_t)dt + dB_t, \hfill M_0 = 0  \\
dY_t = (-b_x(t,X_t^{\ast},u_t^{\ast})^{\top}Y_t-\sigma_x(t,X_t^{\ast},u_t^{\ast})^{\top}Z_t-f_x(t,X_t^{\ast},u_t^{\ast})^{\top})dt\\ 
\hspace{5em}+ Z_t dW_t + \zeta_t dB_t,  \hfill Y_T=(g_x(X_T))^{\top} 
\end{cases}
\end{equation}
where $Z$ is the martingale representation of $Y$ with respect to $W$ and $\zeta$ is the martingale
representation of $Y$ with respect to $B$ \cite{Bao_DCDS_15}.

To solve the data driven feedback optimal control problem, we also use gradient descent type optimization and the gradient process $(J^{\ast})_u^{'}$ is defined in (\ref{PFSGD:EQ5}). Then, we can use the following gradient descent iteration to find the optimal control $u_t^{\ast}$ at any time instant $t \in [0,T]$
\begin{equation}
\label{PFSGD:EQ7}
u_t^{l+1,M} = u_t^{l,M} - r(J^{\ast})_u^{'}(u_t^{l,M}), \;\;\;\; l= 0,1,2,\dots,
\end{equation}
where $r$ is the step-size for the gradient. We know that the observational information $\mathcal{F}_t^M$ is gradually increased as we collect more and more data over time. Therefore, at a certain time instant $t$, we target on finding the optimal control $u_t^{\ast}$ with accessible information $\mathcal{F}_t^M$. Since the evaluation for $(J^{\ast})_u^{'}(u_t^{l,M})$ requires trajectories ${(Y_s,Z_s)}_{t \leq s \leq T}$ as $Y_t$ and $Z_t$ are solved backwards from T to t, we take conditional expectation $E[\cdot |\mathcal{F}_t^M ]$ to the gradient process $\{(J^{\ast})_u^{'}(u_t^{l,M})\}_{t \leq s \leq T}$, i.e.,
\begin{equation}
\label{PFSGD:EQ8}
\begin{aligned}
E[(J^{\ast})_u^{'}(u_s^{l,M})|\mathcal{F}_t^M ] &= E\Big[ b_u(s,X_s,u_s^{l,M})^{\top}Y_s +  \sigma_u(s,X_s,u_s^{l,M})^{\top}Z_s \\ &+ f_u(s,X_s,u_s^{l,M})^{\top} |\mathcal{F}_t^M \Big], \;\;\;\;\;\;\;\; s \in [t,T],
\end{aligned}
\end{equation}
where $X_s$, $Y_s$ and $Z_s$ are corresponding to the estimated control $u_s^{l,M}$. For the
gradient descent iteration (\ref{PFSGD:EQ7}) on the time interval $[t, T ]$, by taking conditional expectation $E[\cdot |\mathcal{F}_t^M ]$, we obtain 
\begin{equation}
\label{PFSGD:EQ9}
E[u_s^{l+1,M}|\mathcal{F}_t^M ] = E[u_s^{l,M}|\mathcal{F}_t^M ] - rE[(J^{\ast})_u^{'}(u_s^{l,M})|\mathcal{F}_t^M ], \;\;\;\; l= 0,1,2,\dots \;\;\; s \in [t,T].
\end{equation}
When $s > t$, the observational information $\{\mathcal{F}_s^M\}_{t \leq s \leq T}$ is not available at time t. We use conditional expectation $E[u_s^{l+1,M}|\mathcal{F}_t^M ]$ to replace $u_s^{l,M}$ since it provides the best approximation for $u_s^{l,M}$ given the current observational information $\mathcal{F}_t^M$. We denote
\begin{equation*}
    u_s^{l,M}|_t := E[u_s^{l+1,M}|\mathcal{F}_t^M ]
\end{equation*}
and then the gradient descent iteration is
\begin{equation}
\label{PFSGD:EQ10}
u_s^{l+1,M}|_t = u_s^{l,M}|_t - rE[(J^{\ast})_u^{'}(u_s^{l,M}|_t)|\mathcal{F}_t^M ], \;\;\;\; l= 0,1,2,\dots \;\;\; s \in [t,T],
\end{equation}
where $E[(J^{\ast})_u^{'}(u_s^{l,M}|_t)|\mathcal{F}_t^M ]$ can be obtained by solving the following FBSDEs
\begin{equation}
\label{PFSGD:EQ11}
\begin{aligned}
dX_t &= b(s,X_s,u_s^{l,M}|_t)ds + \sigma (s,X_s,u_s^{l,M}|_t) dW_s, \hspace{13em} s \in [t,T] \\
dY_t &= (-b_x(s,X_s,u_s^{l,M}|_t)^{\top}Y_s -\sigma_x(s,X_s,u_s^{l,M}|_t)^{\top}Z_s -f_x(s,X_s,u_s^{l,M}|_t)^{\top})ds \\ 
& \hspace{5em}+ Z_s dW_s + \zeta_s dB_s,  \hspace{15em} Y_T=(g_x(X_T))^{\top} 
\end{aligned}
\end{equation}
and evaluated effectively using several well established numerical algorithms \cite{Bao_AA20, Bao_Split, BSDE_filter}. 

When the controlled dynamics and the observation function $\phi$ are nonlinear, we will use optimal filtering techniques to obtain the conditional expectation $E[\Psi(s) |\mathcal{F}_t^M ]$. Before applying the particle filter method, which is one of the most important particle based optimal filtering methods, we define 
\begin{equation}
\label{PFSGD:EQ12}
\Psi(s,X_s,u_s^{l,M}|_t):=b_u^{\top}(s,X_s,u_s^{l,M}|_t) Y_s +  \sigma_u^{\top}(s,X_s,u_s^{l,M}|_t)Z_s
+f_u^{\top}(s,X_s,u_s^{l,M}|_t)
\end{equation}
for $s \in [t,T]$. With the conditional probability density function (pdf) $p(X_t|\mathcal{F}_t^M)$ that we obtain through optimal filtering methods and the fact that $\Psi(s,X_s,u_s^{l,M}|_t)$ is a stochastic process depending on the state of
random variable $X_t$, the conditional gradient process $E[(J^{\ast})_u^{'}(u_t^{l,M}|_t)|\mathcal{F}_t^M ]$ in (\ref{PFSGD:EQ8}) can be obtained by the following integral
\begin{equation}
\label{PFSGD:EQ13}
E[(J^{\ast})_u^{'}(u_s^{l,M}|_t)|\mathcal{F}_t^M ] = \int_{\mathbb{R}^d} E[\Psi(s,X_s,u_s^{l,M}|_t)|X_t=x] \cdot p(x|\mathcal{F}_t^M) dx, \hspace{2em}
s \in [t,T].
\end{equation}

\subsubsection{Numerical Approach for Data Driven Feedback Control by PF-SGD}
For the numerical framework, we need the temporal partition $\Pi_{N_T}$
$$\Pi_{N_T}=\{t_n, 0=t_0 < t_1 < \cdots< t_{N_T}=T\},$$
and we use the control sequence $\{u_{t_n}^{\ast}\}_{n=1}^{N_T}$ to represent the control process $u^{\ast}$ over the time interval $[0, T]$.

\vspace{0.5cm}
\hspace{-0.72cm} \textbf{Numerical Schemes for FBSDEs}

For the FBSDEs system, we adopt the following schemes:
\begin{equation}
\begin{aligned}
\label{PF-FBSDE:EQ14}
&X_{i+1} = X_{i} + b(t_{i},X_{i},u_{t_{i}}^{l,M}|_{t_n})\triangle t_i + \sigma(t_{i},X_{i},u_{t_{i}}^{l,M}|_{t_n})\triangle W_{t_{i}}, \\
&Y_{i} = E_i[Y_{i+1}] + E_i[b_x(t_{i+1},X_{i+1},u_{t_{i+1}}^{l,M}|_{t_n})^{\top}Y_{i+1}\\ 
& \;\;\;\;\;\;\;\;\;\;\;\;\;\;\;\;\;\;\;\;\;\;\; +\sigma_x(t_{i+1},X_{i+1},u_{t_{i+1}}^{l,M}|_{t_n})^{\top}Z_{i+1}+f_x(t_{i+1},X_{i+1},u_{t_{i+1}}^{l,M}|_{t_n})^{\top}] \triangle t_i, \\
&Z_{i} =\frac{1}{\triangle t_i} E_i[Y_{i+1}\triangle W_{t_i}], 
\end{aligned}
\end{equation}
where $X_{i+1}, Y_{i}, \;\text{and} \;Z_{i}$ are numerical approximations for $X_{t_{i+1}}, Y_{t_{i}}, \;\text{and} \; Z_{t_{i}}$, respectively.

Then, the standard Monte Carlo method can approximate expectations with K random samples:
\begin{equation}
\begin{aligned}
\label{PF-FBSDE:EQ15}
&X_{i+1}^k = X_{i} + b(t_{i},X_{i},u_{t_{i}}^{l,M}|_{t_n})\triangle t_i + \sigma(t_{i},X_{i},u_{t_{i}}^{l,M}|_{t_n})\sqrt{\triangle t_i}\omega_i^k, k=1,2, \dots, K,\\
&Y_{i} =\sum_{k=1}^K \frac{Y_{i+1}^k}{K} + \frac{\triangle t_i}{K} \sum_{k=1}^K [b_x(t_{i+1},X_{i+1}^k,u_{t_{i+1}}^{l,M}|_{t_n})^{\top}Y_{i+1}^k\\ 
& \;\;\;\;\;\;\;\;\;\;\;\;\;\;\;\;\;\;\;\;\;\;\; +\sigma_x(t_{i+1},X_{i+1}^k,u_{t_{i+1}}^{l,M}|_{t_n})^{\top}Z_{i+1}^k+f_x(t_{i+1},X_{i+1}^k,u_{t_{i+1}}^{l,M}|_{t_n})^{\top}], \\
&Z_{i} =\frac{1}{\triangle t_i} \sum_{k=1}^K \frac{Y_{i+1}^k \sqrt{\triangle t_i}\omega_i^k}{K},
\end{aligned}
\end{equation}
where $\{ \omega_i^k \}_{k=1}^K$ is a set of random samples following the standard Gaussian distribution that
we use to describe the randomness of $\triangle W_{t_i}$.

The above schemes solve the FBSDE system (\ref{PFSGD:EQ6}) as a recursive algorithm, and the convergence of the schemes is well studied — cf. (\cite{bao2011numerical, Bao_first, Bao_CiCP18}, and \cite{zhao2017high}).

\vspace{0.5cm}
\hspace{-0.72cm} \textbf{Particle Filter Method for Conditional Distribution}

To apply the particle filter method, we consider the controlled process on time interval $[t_{n-1},t_{n}]$
\begin{equation}
\label{PFSGD:EQ14}
X_{t_{n}} = X_{t_{n-1}} + \int_{t_{n-1}}^{t_{n}} b(s,X_s,u_s)ds + \int_{t_{n-1}}^{t_{n}} \sigma(s,X_s,u_s)dW_s
\end{equation}

Assume that at time instant $t_{n-1}$, we have $S$ particles, denoted by $\{x_{n-1}^{(s)}\}_{s=1}^{S}$, that follow an empirical distribution $\pi(X_{t_{n-1}}|\mathcal{F}_{t_{n-1}}^M) := \frac{1}{S}\sum_{s=1}^{S}\delta_{x_{n-1}^{(s)}}(X_{t_{n-1}})$ as an approximation for $p(X_{t_{n-1}}|\mathcal{F}_{t_{n-1}}^M)$. The prior pdf that we want to find in the prediction stage is approximated as
\begin{equation}
\label{PFSGD:EQ15}
\Tilde{\pi}(X_{t_{n}}|\mathcal{F}_{t_{n-1}}^M) := \frac{1}{S}\sum_{s=1}^{S}\delta_{\Tilde{x}_{n}^{(s)}}(X_{t_{n}})
\end{equation}
where $\Tilde{x}_{n}^{(s)}$ is sampled from $\pi(X_{t_{n-1}}|\mathcal{F}_{t_{n-1}}^M) p(X_{t_{n}}|X_{t_{n-1}})$ and $p(X_{t_{n}}|X_{t_{n-1}})$ is the transition probability derived from the state dynamics (\ref{PFSGD:EQ14}). As a result, the sample cloud $\{\Tilde{x}_{n}^{(s)}\}_{s=1}^{S}$ provides an approximate distribution for the prior $p(X_{t_{n}}|\mathcal{F}_{t_{n-1}}^M)$. Then, in the update
stage, we have
\begin{equation}
\label{PFSGD:EQ16}
\Tilde{\pi}(X_{t_{n}}|\mathcal{F}_{t_{n}}^M) :=
\frac{\sum_{s=1}^{S} \delta_{\Tilde{x}_{n}^{(s)}}(X_{t_{n}}) p(M_{t_{n}}|\Tilde{x}_{n}^{(s)})}{\sum_{s=1}^{S}p(M_{t_{n}}|\Tilde{x}_{n}^{(s)})}= \sum_{s=1}^{S} w_{n}^{(s)}  \delta_{\Tilde{x}_{n}^{(s)}}(X_{t_{n}})
\end{equation}
In this way, we obtain a weighted empirical distribution $\Tilde{\pi}(X_{t_{n}}|\mathcal{F}_{t_{n}}^M)$ that approximates the
posterior pdf $p(X_{t_{n}}|\mathcal{F}_{t_{n}}^M)$ with the importance density weight $w_{n}^{(s)} \propto p(M_{t_{n}}|\Tilde{x}_{n}^{(s)})$ \cite{Bao_parameter, Kang-PF, particle-filter}. Then, to avoid 
degeneracy problem, we need the resampling step. So, we have
\begin{equation}
\label{PFSGD:EQ17}
\pi(X_{t_{n}}|\mathcal{F}_{t_{n}}^M) = \frac{1}{S} \sum_{s=1}^{S} \delta_{x_{n}^{(s)}}(X_{t_{n}})
\end{equation}

Then, we combine the numerical schemes for the adjoint FBSDEs
system (\ref{PFSGD:EQ11})  and the particle filter algorithm
to formulate an efficient stochastic optimization algorithm to solve for the optimal control process $u^{\ast}$

\vspace{0.5cm}
\hspace{-0.72cm} \textbf{Stochastic Optimization for Control Process}

In this subsection, we combine the numerical schemes for the adjoint FBSDEs system (\ref{PFSGD:EQ11})  and the particle filter algorithm
to formulate an efficient stochastic optimization algorithm to solve for the optimal control process $u^{\ast}$.

On a time instant $t_n \in \Pi_{N_T}$, we have 
\begin{equation}
\label{PFSGD:EQ--13}
E[(J^{\ast})_u^{'}(u_{t_i}^{l,M}|_{t_{n}})|\mathcal{F}_{t_{n}}^M ] = \int_{\mathbb{R}^d} E[\Psi({t_i},X_{t_i},u_{t_i}^{l,M}|_{t_{n}})|X_{t_{n}}=x] \cdot p(x|\mathcal{F}_{t_{n}}^M) dx,
\end{equation}

where $t_i \geq t_n$ is a time instant after $t_n$.

Then, we use the approximate solutions $(Y_i,Z_i)$ of FBSDEs from schemes (\ref{PF-FBSDE:EQ15}) to replace  $(Y_{t_i},Z_{t_i})$ and the conditional distribution $p(X_{t_{n}}|\mathcal{F}_{t_{n}}^M)$ is approximated by the empirical distribution $\pi (X_{t_{n}}|\mathcal{F}_{t_{n}}^M)$ obtained from the particle filter algorithm (\ref{PFSGD:EQ15}) - (\ref{PFSGD:EQ17}). Then, we can solve for the optimal control $u_{t_n}^{\ast}$ through the following gradient descent optimization iteration
\begin{equation}
\label{PFSGD:EQ--18}
\begin{aligned}
u_{t_i}^{l+1,M}|_{t_n} =u_{t_i}^{l,M}|_{t_n} &-r\frac{1}{S}\sum_{s=1}^{S} E \Big[b_u^{\top}(t_i,X_{t_{i}},u_{t_i}^{l,M}|_{t_n}) Y_{{i}} +  \sigma_u^{\top}(t_i,X_{t_{i}},u_{t_i}^{l,M}|_{t_n})Z_{{i}} \\ 
&+f_u^{\top}(t_i,X_{t_{i}},u_{t_i}^{l,M}|_{t_n})| X_{t_n} = x_n^{(s)}  \Big]
\end{aligned}
\end{equation}
Then, the standard Monte Carlo method can approximate expectation $E \Big[\cdot | X_{t_n} = x_n^{(s)}  \Big]$ by $\Lambda$ samples:
\begin{equation}
\label{PFSGD:EQ--19}
\begin{aligned}
u_{t_i}^{l+1,M}|_{t_n} \approx u_{t_i}^{l,M}|_{t_n} &-r\frac{1}{S}\frac{1}{\Lambda}\sum_{s=1}^{S} \sum_{\lambda=1}^{\Lambda} \Big[b_u^{\top}(t_i,X_{t_{i}}^{(\lambda,s)},u_{t_i}^{l,M}|_{t_n}) Y_{{i}} +  \sigma_u^{\top}(t_i,X_{t_{i}}^{(\lambda,s)},u_{t_i}^{l,M}|_{t_n})Z_{{i}} \\ 
&+f_u^{\top}(t_i,X_{t_{i}}^{(\lambda,s)},u_{t_i}^{l,M}|_{t_n}) | X_{t_n} = x_n^{(s)} \Big]
\end{aligned}
\end{equation}

We can see from the above Monte Carlo approximation that in order to approximate the expectation in one gradient descent iteration step, we need to generate $S \times \Lambda$ samples. This is even more
computationally expensive when the controlled system is a high dimensional process. 

So, we want to apply the idea of stochastic gradient descent (SGD) to improve the efficiency of classic gradient descent optimization and combines it with the particle filter method. Instead of using the fully calculated Monte Carlo simulation to approximate the
conditional expectation, we use only one realization of $X_{t_i}$ to represent the expectation. For the conditional distribution of the controlled process, we can use the particles to describe. So, we have
\begin{equation}
\label{PFSGD:EQ--0018}
\begin{aligned}
E[(J^{\ast})_u^{'}(u_{t_i}^{l,M}|_{t_{n}})|\mathcal{F}_{t_{n}}^M] \approx
b_u^{\top}(t_i,X_{t_{i}}^{(\hat{l},\hat{s})},u_{t_i}^{l,M}|_{t_n}) Y_{{i}}^{(\hat{l},\hat{s})} +  \sigma_u^{\top}(t_i,X_{t_{i}}^{(\hat{l},\hat{s})},u_{t_i}^{l,M}|_{t_n})Z_{{i}}^{(\hat{l},\hat{s})}
+f_u^{\top}(t_i,X_{t_{i}}^{(\hat{l},\hat{s})},u_{t_i}^{l,M}|_{t_n})
\end{aligned}
\end{equation}
where $l$ is iteration index, the index $\hat{l}$ indicates that the random generation of the controlled process varies among the gradient
descent iteration steps .$X_{t_{i}}^{(\hat{l},\hat{s})}$ indicates a randomly generated realization of the controlled process with a randomly selected initial state $X_{t_{n}}^{(\hat{l},\hat{s})} = x_n^{\hat{s}}$ from the particle cloud $\{x_n^{(s)}\}_{s=1}^S$.

Then, we have the following SGD schemes:
\begin{equation}
\label{PFSGD:EQ18}
\begin{aligned}
u_{t_i}^{l+1,M}|_{t_n} =u_{t_i}^{l,M}|_{t_n} &-r \Big(b_u^{\top}(t_i,X_{t_{i}}^{(\hat{l},\hat{s})},u_{t_i}^{l,M}|_{t_n}) Y_{{i}}^{(\hat{l},\hat{s})} +  \sigma_u^{\top}(t_i,X_{t_{i}}^{(\hat{l},\hat{s})},u_{t_i}^{l,M}|_{t_n})Z_{{i}}^{(\hat{l},\hat{s})} \\ 
&+f_u^{\top}(t_i,X_{t_{i}}^{(\hat{l},\hat{s})},u_{t_i}^{l,M}|_{t_n}) \Big)
\end{aligned}
\end{equation}
 $Y_{{i}}^{(\hat{l},\hat{s})}$ is the approximate solution $Y_i$ corresponding to the random sample $X_{t_{i}}^{(\hat{l},\hat{s})}$. And the path of $X_{t_{i}}^{(\hat{l},\hat{s})}$ is generated as following

\begin{equation}
\label{PFSGD:EQ19}
X_{t_{i+1}}^{(\hat{l},\hat{s})} = X_{t_{i}}^{(\hat{l},\hat{s})} + b(t_{i},X_{t_{i}}^{(\hat{l},\hat{s})},u_{t_i}^{l,M}|_{t_n})\triangle t_i + \sigma (t_{i},X_{t_{i}}^{(\hat{l},\hat{s})},u_{t_i}^{l,M}|_{t_n}) \sqrt{\triangle t_i} \omega_i^{(\hat{l},\hat{s})}    
\end{equation}
where $\omega_i^{(\hat{l},\hat{s})} \sim N(0,1)$. Then, an estimate for our desired data driven optimal
control at time instant $t_n$ is
\begin{equation*}
\hat{u}_{t_n} :=   u_{t_n}^{L,M}|_{t_n}
\end{equation*}
The scheme for FBSDE is 
\begin{equation}
\label{PFSGD:EQ20}
\begin{aligned}
&Y_{{i}}^{(\hat{l},\hat{s})} = Y_{{i+1}}^{(\hat{l},\hat{s})} + \Big[b_x(t_{i+1},X_{t_{i+1}}^{(\hat{l},\hat{s})},u_{t_i}^{l,M}|_{t_n})^{\top}Y_{{i+1}}^{(\hat{l},\hat{s})} + \sigma_x(t_{i+1},X_{t_{i+1}}^{(\hat{l},\hat{s})},u_{t_i}^{l,M}|_{t_n})^{\top}Z_{{i+1}}^{(\hat{l},\hat{s})}\\ 
& \;\;\;\;\;\;\;\;\;\;\;\;  + f_x(t_{i+1},X_{t_{i+1}}^{(\hat{l},\hat{s})},u_{t_i}^{l,M}|_{t_n})^{\top}\Big] \triangle t_i\\
&Z_{{i}}^{(\hat{l},\hat{s})} =\frac{Y_{{i+1}}^{(\hat{l},\hat{s})}\omega_i^{(\hat{l},\hat{s})}  }{\sqrt{\triangle t_i}}
\end{aligned}
\end{equation}
Then, we have the the following algorithm:
\begin{algorithm}
\caption{PF-SGD algorithm for data driven feedback control problem}
\label{alg2}
Initialize the particle cloud $\{x_0^{(s)}\}_{s=1}^{S} \sim \xi$ and the number of iteration $L \in \mathbb{N}$
\begin{algorithmic}
\While{$n=0,1,2,\cdots,N_T,$}
    \State Initialize an estimated control process $\{u_{t_i}^{(0,M)}|_{t_n}\}_{i=n}^{N_T}$ and a step-size $\rho$
    \For{SGD iteration steps $l=0,1,2,\cdots,L$,}
        \State Simulate one realization of controlled process $\{X_{t_{i+1}}^{(\hat{l},\hat{s})}|_{t_n}\}_{i=n}^{N_T-1}$ through scheme (\ref{PFSGD:EQ19}) with      
        \State $\{X_{t_{n}}^{(\hat{l},\hat{s})}\}=x_n^{(\hat{s})} \in \{x_n^{(s)}\}_{s=1}^{S}$;
        \State Calculate solution $\{Y_{t_{i}}^{(\hat{l},\hat{s})}|_{t_n}\}_{i=N_T}^{n}$ of the FBSDEs system (14) corresponding to          
        \State $\{X_{t_{i+1}}^{(\hat{l},\hat{s})}|_{t_n}\}_{i=n}^{N_T-1}$ through schemes (\ref{PFSGD:EQ20});
        \State Update the control process to obtain $\{u_{t_i}^{(l+1,M)}|_{t_n}\}_{i=n}^{N_T}$ through scheme (\ref{PFSGD:EQ18});
    \EndFor
    \State The estimated optimal control is given by $\hat{u}_{t_n}^{\ast} = u_{t_n}^{(L,M)}|_{t_n}$
    \State Propagate particles through the particle filter algorithm (\ref{PFSGD:EQ15}) - (\ref{PFSGD:EQ17}) to obtain $\{x_{n+1}^{(s)}\}_{s=1}^{S}$ 
    \State by using the estimated optimal control $\hat{u}_{t_n}^{\ast}$.
\EndWhile
\end{algorithmic}
\end{algorithm}

\eject

\section{Convergence analysis}

The goal of our convergence analysis is to show the convergence of the distribution of the state to the ``true state" under the temporal model discretization $N$. We also show that the estimated control to the ``true control"  To proceed, we first introduce our notations and the assumptions required in the proof in Section \ref{Notations}. Then, in Section \ref{Analysis}, we shall provide the main convergence theorems.

\subsection{Notations and assumptions}\label{Notations}

\textbf{Notations}
\begin{enumerate}
    \item We use $U_n:\lbrace t_n,..., T\rbrace \rightarrow \bR^d $ to denote the control process starting from time $t_n$ and ends at time $T$. 
    We use $$\mathcal{U}_n:=\lbrace U_n | U_n:\lbrace t_n,..., T \rbrace \rightarrow \bR^d, U_n \ \text{is $\calf^M_{t_n}$-adapted}\rbrace$$  to denote the collection of the admissible controls starting at time $t_n$. 
    
    \item We define the control at time $t_n$ to be $u_n:=U_n \big |_{t_n}$, the conditional distribution coming from a particle filter algorithm.  
	\item We define $\mu^N_n:= \pi^N_{t_n|t_n}$ where the superscript means that the measure is obtained through the particle filter method, and so it is random. 
	\item We use $\mathcal{S}^\mathcal{N}$ to denote the sampling operator:  $\pi^N_{t_n|t_n}\triangleq\frac{1}{N}\sum_{i=1}^{N}\delta_{x_t^{(i)}}$ and  $L_n$ to denote the updating step in the particle filter.We use $P^N_n$ to denote the transition operator (the prediction step) under the SGD-particle filter framework. And $P_n$ is the deterministic transition operator for the exact case (the control is exact in SDG). We mention ``deterministic" here to distinguish the case where the control $u_n$ may be random due to the SGD optimization algorithm.

	\item We use $\langle \cdot , \cdot \rangle$ to denote the deterministic $L_2$ inner product, i.e. if $f,g \in L^2([0,T]; \bR^d )$, then 
	\begin{align}
		\langle f , g \rangle := \int^T_0 f\cdot g  \ dt
	\end{align}
	\item We define $J'^x_N(U_n):=\E[J'_N(U_n)| X_n=x]$. 
We then have  $\E[J'^{X_n}_N(U_n)]:= \int \E[J'_N(U_n) \big| X_n=s] \ d \mu^N_n(s) $. We remark that $U_n$ is a process that starts from time $t_n$, and so $X_n$ is essentially the initial condition of the diffusion process. 
	\item We define the distance between two random measures to be the following:
	\begin{align}
		d(\mu,\nu):= \sup_{||f||_{\infty} \leq 1 } \sqrt{\E^{\omega}[|\mu^{\omega}f-\nu^{\omega}f|^2]} \label{dist_random_meas}
	\end{align}
	where the expectation is taken over the randomness of the measure. 
	\item We use the total variation distance between two deterministic probability measures $\mu,\nu$:
	\begin{align}
		d_{TV}(\mu,\nu):=\sup_{||f||_{\infty} \leq 1 } |\mu f-\nu f|
	\end{align}
	\item We use $K_n$ to denote the total number of iterations taken in the SGD algorithm at time $t_n$; we use $\caln$ to denote the total number of particles in the system. We use $C$ to denote a generic constant which may vary from line to line. 
	\item Abusing the notation, we will denote $J'^x_N(U_n)$ in the following way where the argument $U_n$ can be vector of any length $1 \leq n \leq N$:
	\begin{align}
	J'^x_N(U_n)\big|_{t_i}:=\E^{X_{t_n}=x}[f_u'(X_{t_i}, U_n \big|_{t_i} ) + b'^T_u(X_i, U_n \big |_{t_i}) Y_{t_i}]
\end{align}

\end{enumerate}

\textbf{Assumptions}

\begin{enumerate}
	\item We assume that $J'_N$ satisfy the following strong condition:  for any $x \in X$, there exist a constant  $\lambda > 0$ such that for all $U, V \in \calu_0$:
\begin{align}
	\lambda ||U-V||^2 \leq \langle J'^x_N(U)-J'^x_N(V), U-V\rangle  \label{strong_convexity}
\end{align}
Notice that\eqref{strong_convexity}  will imply that  such inequality is true for any $U_n,V_n \in \mathcal{U}_n $. And it can be seen from simply fixing all the $U_{n} |_{t_i}, V_{n} |_{t_i}$, $0 \leq i \leq n-1$ to be 0. 

This is a very strong assumption, one should consider relaxing it to \begin{align}
	\lambda ||U-V||^2 \leq \E^{\omega}[\E^{\mu_n, \cdot }[\langle J'^x_N(U),J'^x_N(V), U-V\rangle] ]\label{strong_convexity_2}
\end{align}
That is, this relation holds in expectation instead of point-wise

   \item both $b$ and $\sigma$ are deterministic and in $C_b^{2,2}(\mathbb{R}^d \times \mathbb{R}^m;\mathbb{R}^d)$ in space variable $x$ and control $u$.  
\item   $b, b_x, b_u, \sigma, \sigma_x, f_x, f_u $ are all uniformly lipschitz in $x,u$ and uniformly bounded.
		\item $\sigma$ satisfies the uniform elliptic condition. 
	\item The initial condition $X_0:=\xi \in L^2(\calf_0)$. 
	\item The terminal (Loss) function $\Phi$ is $C^1$ and positive, and $\Phi_x$ has at most linear growth at infinity.
	\item We assume that the function $g_n$ (related to the Bayesian step) takes has the following bound: there exist $0 < \kappa <1$ such that $$\kappa \leq g_n \leq   \kappa^{-1}$$
\end{enumerate}

\subsection{The convergence theorem for the data driven feedback control algorithm}\label{Analysis}
Our algorithm combines the particle filter method and the stochastic gradient descent method. Lemma \ref{lemma_kody} (combine Lemma 4.7-4.9 from the book \cite{law2015data}) provides the convergence result for the particle filter method alone. It shows that each prediction and updating step is guaranteed to be convergent. 

Recall that $\mathcal{S}^\mathcal{N}$ is the sampling operator where we sample $\mathcal{N}$ particles. $P^N_n$ denotes the transition operator (the prediction step) under the SGD-particle filter framework. $P_n$ denotes the deterministic transition operator assuming that SGD gives the exact control. $L_n$ denotes the updating step in the particle filter method.
\begin{lemma}\label{lemma_kody}
We assume that there exists $\kappa \in (0,1]$. The following is true: 
 \begin{align}
 	\sup_{\mu \in \mathcal{P}(\mu)} d (S^{\mathcal{N}} \mu, \mu ) &\leq \frac{1}{\sqrt{{\mathcal{N}}}} \nonumber \\ 
 	d(P^N_n\mu, P^N_n\nu) &\leq d(\mu,\nu) \\
     d(P_n\mu, P_n\nu) &\leq  d(\mu,\nu) \nonumber\\ 
 	d(L_n \nu, L_n \mu) &\leq 2 \kappa^{-2} d (\nu,\mu)
 \end{align}
\end{lemma}

Given Lemma \ref{lemma_kody}, Theorem 4.5 in \cite{law2015data} tells us the particle filter framework is convergent. Then, following Lemma \ref{lemma_kody}, we can have the distance between the true distribution of the state and the estimated distribution through the SGD-particle filter framework.

\begin{align}
	d(\mu^{N,\cdot}_{n+1}, \mu_{n+1})&\equiv d(L_n S^\mathcal{N} P^N_n \mu^{N,\cdot}_n, L_nP_n\mu_n) \\ 
	& \leq d(L_n S^\mathcal{N} P^N_n \mu^{N,\cdot}_n, L_n S^\mathcal{N} P_n\mu_n)+ d(L_n S^\mathcal{N} P\mu_n ,L_nP_n\mu_n) \nonumber \\
	& \leq 2 \kappa^{-2} \Big( \frac{2}{\sqrt{\mathcal{N}}} +d(P^N_n \mu^{N,\cdot}_n,  P_n\mu^{N,\cdot}_n) + d(S^\mathcal{N} P_n \mu_n^{N,\cdot}, P_n \mu_n) \Big) \nonumber \\ 
	& \leq  2 \kappa^{-2} \Big( \frac{3}{\sqrt{\mathcal{N}}} +d(P^N_n \mu^{N,\cdot}_n,  P_n\mu^{N,\cdot}_n) + d(\mu_n^{N,\cdot}, \mu_n) \Big) \label{first_ineq}
\end{align}
where in the above inequalities, we have used triangle inequalities and lemma \ref{lemma_kody}.

Hence, if we can show that the inequality of the following form holds
\begin{align} \label{intuitive_res}
	d(P^N_n \mu^{N,\cdot}_n,  P_n\mu^{N,\cdot}_n) \leq C_n d(\mu_n^{N,\cdot}, \mu_n)+ \epsilon_n
\end{align}
for some constant $C_n$ and $\epsilon_n$ that we can tune,
then by recursion, we can show that by using \eqref{first_ineq} the convergence holds true.
\begin{remark}
	We point out that the difficulty lies in showing \eqref{intuitive_res}. 
	Recall that the distance between two random measures is defined in \eqref{dist_random_meas} which involves testing the overall measurable function bounded by $1$. However, we will see later that it is more desirable to test against Lipschitz functions. Hence, since the underlying measure is a finite Borel probability measure, we want to identify the function first with a continuous function on a compact set (\textit{Lusin}). Then we approximate this continuous function uniformly by a Lipschitz function since now the domain is compact. This way, we can roughly show that a form close to \eqref{intuitive_res} is true. 
	\end{remark}

\begin{remark}
Notice that the first measure in $d(P^N_n \mu^{N,\cdot}_n,  P_n\mu^{N,\cdot}_n)$ has two sources of randomness: the randomness in $P_n$ which comes from the SGD algorithm used to find the control, and the randomness in the measure $\mu^{N}_n$. However, we do not distinguish the two when we take the expectation. 
\end{remark}

To proof the convergence, we need to create a subspace where all particles $X_i$ (obtained from the particle filter method) at any time $n$ are within this bounded subspace. Or we can relax it to the statement that the probability of any particles $X_i$ escaping from a very large region is very small. Lemma \ref{key_bound} shows we can restrict the particles to a compact subspace with the radius $M$ by starting from any particels and any admissible control $U_0$. 

\begin{lemma} \label{key_bound}
	There exists M and constant $C$, such that under any admissible control $U_0$
	\begin{align} 
		\mathbb{P}(\sup_{\substack{ {i \in \lbrace 1, ... , N \rbrace}\\ {U_0 \in \mathcal{U}}}}|X_i| \geq M) \leq \frac{C}{M^2}, \ \ X_i \sim \pi^N_{t_i | t_{i-1}} \ \textit{or} \ \   X_i \sim \pi^N_{t_i | t_i}
	\end{align}
\end{lemma}

\begin{proof}
We start with time $t_0$. 

\textbf{Step 1.}
	Starting from $X_0 \sim \xi$ with $\E[\xi^2] \leq C_{0}$, and by fixing an arbitrary control $u_0$ we have for the prediction step: 
	\begin{align}
		\E[|X^-_{1}|^2] & =  \E[|X_0+b(X_0, u_0)\Delta t + \sigma(X_0) \Delta W_0|^2] \nonumber \\ 
		& \leq \E[(1+\Delta t) X^2_0+(1+\frac{1}{\Delta t})b^2 (\Delta t)^2]+C^2_{\sigma}\Delta t \nonumber \\ 
		& \leq (1+\Delta t)C^2_0 + (C^2_b(\Delta t +1) + C^2_{\sigma})\Delta t \nonumber \\
		& := C^-_0 \label{random_bound1}
	\end{align}
	
\textbf{Step 2.} We denote the distribution $\mathcal{L}(X^-_1) \sim \pi_{t_1 |t_0}$, then the particle method will do a random resampling from such distribution, and obtain a random distribution 
\begin{align}
	\frac{1}{\caln} \sum^\caln_{i=1} \delta_{x_i(\omega)}: = \pi^N_{t_1 |t_0}
\end{align}
Hence, we have for $X \sim \pi^N_{t_1 |t_0}$, take expectation where the expectation is taken over all randomness in the measure
\begin{align}
	\E[ X^2 ]&=\E \big[\E[ X^2  \big| \mathcal{G}^- _1] \big ]=\frac{\caln}{\caln} \E[\E[x_1^2| \mathcal{G}^- _1]]=\E[\tilde{X}^2]
\end{align}
where $x_i \sim \pi_{t_1 |t_0}$ are i.i.d random samples, $\mathcal{G}_0$ contains the sampling randomness, and $\tilde{X} \sim \pi_{t_1 |t_0}$. The conditional expectation is meant to show that all the particles $\xi$ are conditionally independent (since there are other randomness accumulated in the history if we want to apply this argument recursively.) 
\
Thus, by \eqref{random_bound1}
\begin{align}
	\E[\tilde{X}^2]& \leq C^-_0
\end{align}

\textbf{Step 3.} We now have the random measure $\pi^N_{t_1 |t_0}$, and we proceed to the analysis step. 
We have by definition 
	\begin{align}
		X_1 \sim \frac{g(x) d \pi^N_{t_1 |t_0}(x)}{\int g(x) d \pi^N_{t_1 |t_0}(x)}:= \tilde{\pi}^N_{t_1|t_1} \label{analysis}
	\end{align}
where $\pi^N_{t_1 |t_0}(x)$ is the distribution of the terminal state $X^-_1$ from the previous step. We give an estimate over $\E[|X_1|^2]$: 
\begin{align}
	\E[|X_1|^2] &\leq (\frac{1}{\kappa})^2 \E[\int x^2 d \pi^N_{t_1 |t_0}(x)]  \leq (\frac{1}{\kappa})^2 C_0^-  := C_1 \label{estimate_bound2}
\end{align}

\textbf{Step 4.} Now we again apply the random sampling step
\begin{align}
	\frac{1}{\caln} \sum^\caln_{i=1} \delta_{x_i(\omega)}: = \pi^N_{t_1 |t_1}
\end{align}
where $x_i(\omega) \sim \tilde{\pi}^N_{t_1|t_1}$. Then, for $X \sim \pi^N_{t_1 |t_1}$, we have 
\begin{align}
	\E[ X^2 ]&=\E \big[\E[ X^2  \big| \mathcal{G}_1] \big ] \\ 
	&=\frac{\caln}{\caln} \E[\E[x_i^2\big| \mathcal{G}_1]] \nonumber \\
	&=\E[\tilde{X}^2]
\end{align}
where $\tilde{X} \sim \tilde{\pi}^N_{t_1|t_1}$ and $\mathcal{G}_1$ is the Filtration that builds on $\mathcal{G}^-_1$ and the randomness of the current sampling. Then, by \eqref{estimate_bound2}, we have 
\begin{align}
	\E[ X^2 ] \leq C_1, \ \  X \sim \tilde{\pi}^N_{t_1|t_1}
\end{align}
And this completes all the estimates for the first time-stepping.
Hence, after one time step, we have 
\begin{align}
	C_1=\kappa^{-2}(1+\Delta t) C_0 + C \Delta t
\end{align}
which means that by applying the same argument, we will have the following recursion in general: 
\begin{equation}
	C_{n+1}=\kappa^{-2}(1+\Delta t) C_{n} + C \Delta t
\end{equation}
 As a result, by picking arbitrary $u_n$, using this same argument repeatedly until $N$, we have that for all $n=1,..., N$: 
	 \begin{align}
	 	C_n= (\kappa^{-2}(1+\Delta t))^n C_0 + \sum^{n-1}_{i=0}(\kappa^{-2}(1+\Delta t))^i C \Delta t
	 \end{align} 
And we notice that $C_n$ is increasing in $n$. As a result, we know that for any $X_n \sim \pi^N_{t_n| t_{n-1}}, \pi^N_{t_n| t_n}$, we have that 
\begin{align}
	\E[|X_n|^2] \leq C_N
\end{align}
Hence, by the Chebyshev's inequality, we have 
\begin{align}
	\mathbb{P}(|X_n|\geq M) \leq \frac{C_N}{M^2}, \ \  \forall n \in \lbrace1, 2,...,N \rbrace
\end{align}
then we have that
\begin{align}
	\mathbb{P}(\sup_n|X_n| \geq M) \leq \frac{C_N}{M^2}
\end{align}
By noticing that since the control values are arbitrarily picked, we have that 
\begin{align}
	\mathbb{P}(\sup_{n,U_0}|X_n| \geq M) \leq \frac{C}{M^2},  \ \ X_n \sim \pi^N_{t_n| t_{n-1}} \textit{or} \ \  \pi^N_{t_n| t_n}
\end{align}
\end{proof}

\begin{remark}
Lemma \ref{key_bound} tells that starting
from any random selection of particles and any admissible control $U_0$, at any time $t$, all particles are restricted by a compact set $\calm$ with diameter $diam(\calm) \leq M$, such that 
\begin{align} 
		\mathbb{P}(\sup_{\substack{ {i \in \lbrace 1, ... , N \rbrace}\\ {U_0 \in \mathcal{U}}}}|X_i| \geq diam(\calm )) \leq \frac{C}{M^2}, \ \ X_i \sim \pi^N_{t_i | t_{i-1}} \ \textit{or} \ \   X_i \sim \pi^N_{t_i | t_i}
	\end{align}	
	We will use the following result extensively later
	\begin{align}
		\E[\mathbf{1}_{\lbrace |X_n| \geq M \rbrace}] \leq \frac{C}{M^2}, \ \forall 1 \leq n  \leq N
	\end{align}
\end{remark}

The following lemma \ref{SGD_proof} describes the difference between the estimated optimal control $u_n$ and the true control $u^*_n$.
Let $\mathcal{G}_k:=\lbrace \Delta W_n^i, x^i \rbrace^{k-1}_{i=0}$. We can see that knowing $\calg_k$ essentially means that we know the control $u_k$ in the SGD framework at time $t_n$, since according to our scheme the control is $\calg_k$ measurable. 

\begin{lemma}\label{SGD_proof}
Under a fixed temporal discretization number $N$, with the particle cloud $\mu^{N,\omega}$, a deterministic $u^*_n$ and a compact domain $\calk'_n$, (such that $\E^{\omega} \E^{\mu^{N,\omega}_n}[\mathbf{1}_{\calk'^c_n}] \leq \frac{C}{M^2_n}$ and $diam(\calk'^c_n) \leq M_n$), we have for any iteration number $K$, the following hold
\begin{align}
	\E^{\omega} \Bigg[ \E^{\mu^{N,\omega}_n} \Big[ \mathbf{1}_{\calk'_n}|u^{\omega}_n - u^*_n|^2 \big| X_n=x  \Big] \Bigg] \leq  CM_n^2 \sup_{||q||_{\infty} \leq 1}\E^{\omega}[|\ \mu^{N,\omega}_n q - \mu_n q \ |^2]+\frac{C}{M_n} +\frac{CM_n^2}{K} \label{inter_control_conv}
\end{align}
\end{lemma}

Remark: The value of $\sup_{||q||_{\infty} \leq 1}\E^{\omega}[|\ \mu^{N,\omega}_n q - \mu_n q \ |^2]$ depends on $\mu^{N,\omega}_n$ which is obtained from the previous step, and it does not depend on the current $M_n$. As a result, we can see that, as long as $$\sup_{||q||_{\infty} \leq 1}\E^{\omega}[|\ \mu^{N,\omega}_n q - \mu_n q \ |^2] \rightarrow 0$$ \eqref{inter_control_conv} can be made arbitrarily small on any compact domain $\calk'_n$ And this indicates the point-wise convergence for $u$ at any time $t_n$.
\begin{proof}
For simplicity in the proof, we denote control as $U^{K+1}_n$ where $K$ is the iteration in SGD, and $n$ is current time $t_n$. $j_n'^{x^k}(U^K_n)$ denote the SGD process using estimated control $U^K_n$, $J'^{x}_N(U^*_n)$ denote the process using true control $U^*_n$.
	\begin{align}
		U^{K+1}_n &=U^K_n - \eta_k j_n'^{x^k}(U^K_n) \label{opt1} \\ 
		U^{*}_n &=U^*_n - \eta_k \E^{\mu_n}[J'^{x}_N(U^*_n)] \label{opt2}
	\end{align}
	where $x^k$ is drawn from the current distribution $\mu^{N,\omega}_n$ and $\E^{\mu_n}[J'^{x}_N(U^*_n)]=0$ by the optimality condition. Take the difference between \eqref{opt1} and \eqref{opt2}, square both sides and take conditional expectation $\E[ \cdot | \calg_K]$, and this conditional expectation is taken with respect to two randomness: 
	\begin{enumerate}
		\item The randomness coming from the selection on the initial point $x^k_n$. 
		\item The randomness coming from the pathwise approximated Brownian motion used for FBSDE. 
		\item The randomness coming from the accumulation of the past particle sampling.
	\end{enumerate} 	
We can write $\E[ j'^{x}(U^K) | \calg_K] = \E^{\mu^{N,\omega}_n}[J'^{x}_N(U^K) | \calg_K]$, which can be seen from the following 
	\begin{align}
		\E[j'^{x}(U^K) | \calg_K] &=\E^{\mu^{N,\omega}_n}[ \E^{X_{t_n}=x}[j'^{x}(U^K) \Big| \calg_K ] ] \\ 
		&= \E^{\mu^{N,\omega}_n}[J'^{x}_N(U^K) | \calg_K]
	\end{align}

Then, by taking the square norm on both sides, multiply by an indicator function $\mathbf{1}_{\calk'_n}$ and take conditional expectation $\E[\cdot |\calg_K]$ , noticing that $U_n^K$ is $\calg_K$ measurable and $U^*$ is deterministic in this case, we get 
\begin{small}

\begin{align}
	\E[\mathbf{1}_{\calk'_n}||U_n^{K+1}-U_n^{*}||^2 \Big| \calg_K ]&=\E[\mathbf{1}_{\calk'_n}||U^K_n-U_n^*||^2 \Big| \calg_K]-2\eta_k \langle \E^{\mu^{N,\omega}_n}[\mathbf{1}_{\calk'_n}J'^x_N(U^K_n)\Big| \calg_K] - \E^{\mu_n}[\mathbf{1}_{\calk'_n}J'^{x}_N(U_n^*)], U^K_n-U_n^*\rangle \nonumber \\
&+ \eta_K^2 \E[\mathbf{1}_{\calk'_n}||j'^{x}(U^K_n)- \E^{\mu_n}[J'^x_N(U_n^*)]|| \Big| \mathcal{G}_K  ] \\ 
	&=\E[\mathbf{1}_{\calk'_n}||U^K_n-U_n^*||^2 \Big| \calg_K ]-2\eta_k \E^{\mu^{N,\omega}_n}\Big[\mathbf{1}_{\calk'_n}\langle J'^x_N(U^K_n) -J'^x_N(U_n^*) \nonumber \\
	 &+J'^x_N(U_n^*)-\E^{\mu_n}[J'^x_N(U_n^*)], U^K_n-U_n^*\rangle\Big| \calg_K \Big]+ \eta_k^2 \E^{\mu^N_n}[\mathbf{1}_{\calk'_n}||j'^{x}(U^K_n)- \E^{\mu_n}[J'^x_N(U_n^*)]|| ] \nonumber \\ 
	 & = \E[\mathbf{1}_{\calk'_n}||U^K_n-U_n^*||^2 \Big| \calg_K  ]-2\eta_k \E^{\mu^{N,\omega}_n}[\langle \mathbf{1}_{\calk'_n}J'^x_N(U^K_n) -\mathbf{1}_{\calk'_n}J'^x_N(U^*),U^K_n-U_n^* \rangle \Big|\calg_K]    \nonumber\\
	 &  -2\eta_k\langle \E^{\mu^{N,\omega}_n}[\mathbf{1}_{\calk'_n}J'^x_N(U_n^*)]-\E^{\mu^{N,\omega}_n}[\mathbf{1}_{\calk'_n}\E^{\mu_n}[J'^x_N(U_n^*)]], \mathbf{1}_{\calk'_n}( U^K_n-U_n^*)\rangle \nonumber \\ 
	 &+ \eta_k^2 \E[\mathbf{1}_{\calk'_n}||j'^{x}(U^K_n)- \E^{\mu_n}[J'^x_N(U_n^*)]|| \Big| \calg_K] \nonumber    \\ 
	 & \leq (1-\lambda \eta_k)\E^{\mu^{N,\omega}_n}[\mathbf{1}_{\calk'_n}||U^K_n-U_n^*||^2 \Big| \calg_K] + \frac{\eta_k}{\lambda} \underbrace{|| \E^{\mu^{N,\omega}_n}\Big[\mathbf{1}_{\calk'_n}J'^x_N(U_n^*)-\mathbf{1}_{\calk'_n}\E^{\mu_n}[J'^x_N(U_n^*)] \Big]||^2}_{\mathlarger{**}} \nonumber \\
	 & +\eta_k^2 \E^{\mu^{N,\omega}_n}[\mathbf{1}_{\calk'_n}C|x_i|^2+C] \nonumber
	 	\end{align} 
\end{small}
where in the last line we used Lemma \ref{j_bound}.

Recall that $\E^{\mu_n}[J'^x_N(U^*)]=0$, we then have 
\begin{align}
	** & \leq || \E^{\mu^{N,\omega}_n}\big[\mathbf{1}_{\calk'_n}J'^x_N(U^*) \big]-\E^{\mu^{N,\omega}_n}\big[\E^{\mu_n}[J'^x_N(U^*)] \big]+\E^{\mu^{N,\omega}_n}\big[\E^{\mu_n}[J'^x_N(U^*)] \big]-\E^{\mu^{N,\omega}_n}\big[\mathbf{1}_{\calk'_n}\E^{\mu_n}[J'^x_N(U^*)] \big]||^2 \\
	& \leq || \E^{\mu^{N,\omega}_n}\big[\mathbf{1}_{\calk'_n}J'^x_N(U^*) \big]-\E^{\mu_n}[\mathbf{1}_{\calk'_n}J'^x_N(U^*)]+\E^{\mu_n}[\mathbf{1}_{\calk'_n}J'^x_N(U^*)]-\E^{\mu^{N,\omega}_n}\big[\E^{\mu_n}[J'^x_N(U^*)] ||^2  \nonumber \\
	& \leq (1+\epsilon)|| \E^{\mu^{N,\omega}_n}\big[\mathbf{1}_{\calk'_n}J'^x_N(U^*) \big]-\E^{\mu_n}[\mathbf{1}_{\calk'_n}J'^x_N(U^*)]||^2+(1+\frac{1}{\epsilon})||\E^{\mu_n}[\mathbf{1}_{\calk'_n}J'^x_N(U^*)]-\E^{\mu_n}[J'^x_N(U^*)] ||^2 \nonumber\\
	& \leq C|| \E^{\mu^{N,\omega}_n}\big[\mathbf{1}_{\calk'_n}J'^x_N(U^*) \big]-\E^{\mu_n}[\mathbf{1}_{\calk'_n}J'^x_N(U^*)]||^2+\frac{C}{M_n} 
\end{align}	

Then we take expectation on both sides over the randomness and we have
\begin{align}
    \E[\mathbf{1}_{\calk'_n}||U^{K+1}-U^{*}||^2 ]& \leq (1-\lambda \eta_k)\E^{\mu^N_n}[||U^k-U^*||^2  ]+ \frac{\eta_k}{\lambda}\E^{\omega}[ C|| \E^{\mu^{N,\omega}_n}\big[\mathbf{1}_{\calk'_n}J'^x_N(U^*) \big]-\E^{\mu_n}[\mathbf{1}_{\calk'_n}J'^x_N(U^*)]||^2 \nonumber \\ 
    &+\frac{C}{M_n} ]+\eta_k^2 CM_n^2 \nonumber \\
	 & \leq \frac{||U^0-U^*||^2}{K}+C\E^{\omega}[|| \E^{\mu^{N,\omega}_n} [\mathbf{1}_{\calk'_n}J'^x_N(U^*)]-\E^{\mu_n}[\mathbf{1}_{\calk'_n}J'^x_N(U^*) ||_2^2]+\frac{C}{M_n} +\frac{CM_n^2}{K} \label{opt_end}
\end{align}

Notice that for the control $U^*$ we know that for for a fixed $x$, $J'^x_N(U^*)$ is uniformly bounded: 
\begin{align}
	& \quad ||\E^{\mu^{N,\omega}_n} [\mathbf{1}_{\calk'_n}J'^x_N(U^*)]-\E^{\mu_n}[\mathbf{1}_{\calk'_n}J'^x_N(U^*)] ||^2_2 \\
	&= \sum^N_{i=n} \Delta t \bigg|\E^{\mu^{N,\omega}_n}[\mathbf{1}_{\calk'_n}J'^x_N(U^*)\big|_i]-\E^{\mu_n}[\mathbf{1}_{\calk'_n}J'^x_N(U^*)\big|_i] \bigg|^2 \nonumber \\ 
	& \leq \sum^N_{i=n} \Delta t \sup_{j \in\lbrace n,...,N\rbrace} \bigg|\E^{\mu^{N,\omega}_n} [\mathbf{1}_{\calk'_n}J_N'(U^*)(x)\big|_j]-\E^{\mu_n}[\mathbf{1}_{\calk'_n}J_N'(U^*)(x)\big|_j] \bigg|^2  \nonumber \\
	& \leq \sup_{j \in\lbrace n,...,N\rbrace} \bigg|\E^{\mu^{N,\omega}_n} [\mathbf{1}_{\calk'_n}J'^x_N(U^*)\big|_j]-\E^{\mu_n}[\mathbf{1}_{\calk'_n}J'^x_N(U^*)\big|_j] \bigg|^2
\end{align}
However, since by Lemma \ref{j_bound}:
\begin{align}
	\sup_{j \in\lbrace n,...,N\rbrace}\Big|J'^x_N(U^*)\big|_j \Big|^2 &\leq C|x|^2 +C 
\end{align} 
we have that 
\begin{align}
	\mathbf{1}_{\calk'_n}\sup_{j \in\lbrace n,...,N\rbrace}\Big|J'^x_N(U^*)\big|_j \Big|^2&\leq C|M_n|^2 \ \mathbf{1}_{\calk'_n}|q(x)|
\end{align}
	for some $q(x)$, where $ ||q(x)||_{\infty} \leq 1 $. 
As a result, we see that 
\begin{align}
	\eqref{opt_end} &\leq  CM_n^2\E^{\omega}\big[\Big|\E^{\mu^{N,\omega}_n} [q(x)]-\E^{\mu_n}[q(x)]\Big|^2 \big] +\frac{CM_n^2}{K} \label{opt_estimate1} \\
		&\leq  CM_n^2 \sup_{||q||_{\infty} \leq 1} \E^{\omega}[|\ \mu^{N,\omega}_n q - \mu_n q \ |^2] +\frac{CM_n^2}{K} \label{opt_estimate}	
\end{align}

Thus, we have  
\begin{align}
	\E^{\omega}\Big[\E^{\mu^{N,\omega}_n}[\mathbf{1}_{\calk'_n}\sup_{n}|u^{K+1}-u^{*}|^2 \big] \Big]&\leq  CM_n^2 \sup_{||q||_{\infty} \leq 1}\E^{\omega}[|\ \mu^{N,\omega}_n q - \mu_n q \ |^2] + \frac{C}{M_n}+\frac{CM_n^2}{K}
\end{align}
where we have absorbed the constant term $N$ in $C$. 
\end{proof}
Lemma \ref{SGD_proof} shows that when the empirical distribution $\mu^{N}_n$ is close enough to the true distribution $\mu_n$, the difference between $u_n$ and $u_n^*$ under the expectation restricting to a compact set will be very small.


Next, we want to show that by moving forward in one step, the distance between the true distribution of the state and the estimated distribution through the SGD-particle filter framework is bounded by the distance of the previous step with some constants in Lemma \ref{onestep}.

\begin{lemma}\label{onestep}
For each $n=0,1,...,N-1$, there exist $M_n, L_n, \delta_n, K_n$ such that the following inequality holds
	\begin{align}
		d(\mu^{N,\cdot}_{n+1}, \mu_{n+1}) & \leq 2 \kappa^{-2} \Big( (1+C \Delta t L_n M_n) d(\mu_n^{N, \cdot}, \mu_n) + \frac{C}{M_n} + \frac{C M_n}{K_n} +2 \delta_n + \frac{3}{\sqrt{\caln}})
	\end{align}
\end{lemma}

\begin{proof}
The key step is to estimate the quantity $d^2(P^N_n \mu^{N,\cdot}_n,  P_n\mu^{N,\cdot}_n)$ in \eqref{first_ineq}. 
WLOG, we assume that the sup is realized by the function $f$ with $||f||_{\infty}\leq 1$, then we have 
\begin{align}
	d^2(P^N_n \mu^{N,\cdot}_n,  P_n\mu^{N,\cdot}_n)=\E^{\omega}[|P^N_n \mu^{N,\cdot}_n f - P_n \mu^{N,\cdot}_n f |^2 ]
\end{align}
Notice that $P^N_n$ is the prediction operator that uses the control $u_n$ which carries the randomness from SGD, and $P_n$ uses the control $u^*_n$. Then $P^N_n \mu^{N, \omega}_N$ is a random measure. And we comment that both $u^*_n$ and $\mu_n$ are deterministic. 

Without loss of generality, we use $u_n^{\omega}$ and $\mu^{N,\omega}_n$ to denote  the random control and the random measure. (Even though the randomness can be different, we can concatenate $(\omega_1, \omega_2):=\omega$ to define them as $\omega$ in general.)

 We have for the fixed randomness $\omega$, and by Fubini's theorem
\begin{align}
	|P^N_n \mu^{N,\omega}_n f - P_n \mu^{N,\omega}_n f |^2& = \bigg|\E^{\mu^{N,\omega}_n} \Big[ \E[\underbrace{f(X_n+b(X_n, u^{\omega}_n)\Delta t + \sigma(X_n) \Delta W_n) }_{f^{\omega}_1}    \big| X_n=x ] \Big] \nonumber \\
	&-\E^{\mu^{N,\omega}_n} \Big[ \E[\underbrace{f(X_n+b(X_n, u^*_n)\Delta t+ \sigma(X_n) \Delta W_n)}_{f_2}    \big| X_n=x ] \Big] \bigg|^2 \\
	&=\bigg| \E^{\mu^{N,\omega}_n} \Big[  \E[f^{\omega}_1-f_2  \big| X_n=x ]   \Big]   \bigg|^2 \nonumber \\
	&=\underbrace{\bigg|  \E^{\mu^{N,\omega}_n} \Big[ \mathbf{1}_{\calm_n} \E[f^{\omega}_1-f_2  \big| X_n=x ]   \Big]   \bigg|^2}_{\mathlarger{A_1}} +\underbrace{\bigg|  \E^{\mu^{N,\omega}_n} \Big[\mathbf{1}_{\calm^c_n}  \E[f^{\omega}_1-f_2  \big| X_n=x ]   \Big]   \bigg|^2}_{\mathlarger{A_2}} \label{diff_est1}
\end{align}
where the inner conditional expectation is taken with respect to $\Delta W_n$.

Now, since we can pick $\calm_n$ to be a large compact set containing the origin, with 
\begin{align}
	\mathbb{P}(\sup_{n,U_0}|X_n| \geq \textit{diam}(\calm_n)) \leq \frac{C}{M_n^2}
\end{align}
 To deal with $A_1, A_2$, we see that it is desirable that the function $f$ has the \textit{Lipchitz property}. However, it is only in general measurable. The strategy to overcome this difficulty is to first use the \textit{Lusin's Theorem} to find a continuous identification $\tilde{f}$ with $f$ on a large compact set, then on this compact set, we can approximate $\tilde{f}$ uniformly by a Lipchitz function.

We see that 
\begin{align}
	A_1 &\leq \E^{\mu^{N, \omega}_n} \Big[\mathbf{1}_{\calm_n}  \E[|f^{\omega}_1-f_2 |^2 \big| X_n=x ] \Big]
\end{align}
Then, by taking expectation on both side over all the randomness in this quantity, we have

\begin{align}
	\E^{\omega}[A_1] & \leq \E^{\omega}\E^{\mu^{N,\omega}_n} \Big[\mathbf{1}_{\calm_n}  \E^{}[|f^{\omega}_1-f_2 |^2 \big ] \Big]
\end{align}

We know that there exists a big compact $\calk_n$ (so a large $M_n$) containing the origin such that 
\begin{align}
	\mathbb{P}(\sup_{n,U_0}|X_n| \geq \textit{diam}(\calk_n)) \leq \frac{C}{M_n^2}
\end{align}
and a continuous $\tilde{f}^n$ with $\tilde{f}^n |_{\calk_n}=f|_{\calk_n}$ by \textit{Lusin's theorem}.

And so we know that $\tilde{f}^n | _{\calk_n \cap \calm_n} = f|_{\calk_n \cap \calm_n}$. 
And we also have the following inequality: 
\begin{align}
	\E^{\omega}\E^{\mu^{N,\omega}_n} \Big[\mathbf{1}_{\calm_n}  \E^{}[|f^{\omega}_1-f_2 |^2 \big ] \Big]& = \E^{\omega}\E^{\mu^{N,\omega}_n} \Big[(\mathbf{1}_{\calm_n \cap \calk_n}+ \mathbf{1}_{\calm_n \cap \calk^c_n}) \E^{}[|f^{\omega}_1-f_2 |^2 \big ] \Big] \\
	& \leq \E^{\omega}\E^{\mu^{N,\omega}_n} \Big[(\mathbf{1}_{\calm_n \cap \calk_n}+ \mathbf{1}_{ \calk^c_n}) \E^{}[|f^{\omega}_1-f_2 |^2 \big ] \Big] \label{2nd_bound}
\end{align}

Also, since both $\calk_n$ and $\calm_n$ are compact, $ \calk'_n:=\calk_n \cap \calm_n$ is also compact with $diam(\calk'_n) \leq M_n$.
From Lemma \ref{key_bound}, we know that there exist some constant $C$ such that for any $\pi^N_{t_n | t_{n-1}}$, $\pi^N_{t_n | t_{n}}$ that one obtains from or particle filter-SGD algorithm,  $X \sim \pi^N_{t_n | t_{n-1}}$ or $\pi^N_{t_n | t_{n}}$: 
\begin{align}
	\E^{\omega}\Big[\E[ \mathbf{1}_{\lbrace X \in \calk^c_n \rbrace}]\Big] \leq \frac{C}{M^2_n}
\end{align} 

Hence, we have that 
\begin{align}
	\eqref{2nd_bound} & \leq E^{\omega}\E^{\mu^{N,\omega}_n} \Big[\mathbf{1}_{\calk'_n} \E^{}[|f^{\omega}_1-f_2 |^2 \big ] \Big] +  \frac{C}{M_n^2}
	\end{align}
To deal with $A_2$, notice that $|f^{\omega}_1-f_2| \leq 2 $ by the choice of $f$, we have the following. 
\begin{align}
	A_2 & \leq  \E^{\mu^{N,\omega}_n} \Big[\mathbf{1}_{\calm^c_n}  \E[\big| f^{\omega}_1-f_2 \big|^2 \big| X_n=x ]   \Big]    \nonumber \\ 
	\Rightarrow \E^{\omega}[A_2] &\leq 4\E^{\omega}[\E^{\mu^{N,\omega}_n}[\mathbf{1}_{\calk^c_n} ]] \nonumber \\ 
	&\leq \frac{C}{M^2_n}
\end{align}
by Lemma \ref{key_bound}. 

To deal with $A_1$, we have by the density of the Lipchitz function  there exists 
$||f^n-\tilde{f}^n||_{\calk'_n, \infty} \leq \delta_n$ with Lipschitz constant $L_n$. We point out that $L_n$ may depend on $\calk'_n$, $\delta_n$ and the function $\tilde{f}|_{\calk'_n}$. 
Now by taking the expectation on both sides and using the Lipchitz property, we have
\begin{align}
	 \E^{\omega}[A_1] 
	& \leq (C\Delta t L_n)^2  \underbrace{\E^{\omega} \Bigg[ \E^{\mu^{N,\omega}_n} \Big[ \mathbf{1}_{\calk'_n}|u^{\omega}_n - u^*_n|^2 \big| X_n=x    \Big] \Bigg]}_{\mathlarger{\ast}} +\frac{C}{M^2_n} +4 \delta^2_n
\end{align}

We realize that $\ast$ is the SGD optimization part of the algorithm in expectation, and we note that we have dropped the inner expectation. The  expectation $\E^{\mu^{N,\omega}_n}[\cdot]$ mean that given the initial condition $  X_n =x \in \mathbf{1}_{\calk'_n}$, with $X_n \sim \mu^{N,\omega}_n$, one wants to find the difference in expectation between $u_n$ and $u^*_n$. And the outer expectation $\E^{\omega}[\cdot]$ means averaging overall the randomness in both the measure and the SGD. 
 
Now, by using \eqref{opt_estimate} in Lemma \ref{SGD_proof}, absorbing N in the constant C, we obtain the following 
\begin{align}
	\E^{\omega}[A_1] \leq (C\Delta t L_n)^2NM_n^2\ \sup_{||q|| \leq 1}\E^{\omega}[|\ \mu^{N,\omega}_n q - \mu_n q \ |^2] +\frac{CM_n^2}{K_n} + \frac{C}{M_n^2} + 4 \delta_n^2
\end{align}
By definition of the distance between two random measures, we have that : 
\begin{align}
	\E[A_1] &\leq (C \Delta t L_n)^2 N M_n^2 d^2(\mu_n^{N,\cdot}, \mu_n) +\frac{CM_n^2}{K_n} + \frac{C}{M_n^2} + 4 \delta_n^2 \\ 
	\Rightarrow \sqrt{\E[A_1]} & \leq C \Delta t L_n M_n d(\mu_n^{N,\cdot}, \mu_n)+\frac{C M_n}{\sqrt{K_n}} + \frac{C}{M_n} + 2 \delta_n
\end{align}
Since $\sqrt{\E[A_2]} \leq  \frac{C}{M_n}$, we have that 
\begin{align}
	\eqref{first_ineq} &\leq 2 \kappa^{-2} \Big( \frac{3}{\sqrt{\mathcal{N}}} +C \Delta t L_n M_n d(\mu_n^{N,\cdot}, \mu_n)+\frac{C  M_n}{\sqrt{K_n}} + \frac{C}{M_n} + 2 \delta_n + \frac{2}{M_n} + d(\mu_n^{N,\cdot}, \mu_n) \Big) \nonumber \\ 
	\Rightarrow d(\mu^{N,\cdot}_{n+1}, \mu_{n+1}) & \leq 2 \kappa^{-2} \Big( (1+C \Delta t L_n M_n) d(\mu_n^{N, \cdot}, \mu_n) + \frac{C}{M_n} + \frac{C M_n}{\sqrt{K_n}} +2 \delta_n + \frac{3}{\sqrt{\caln}}) \label{recursive_start}
\end{align}
where in \eqref{recursive_start}, we have merged $\sqrt{N}$ into $C$. 
\end{proof}
\begin{remark}
Lusin's theorem requires the underlying measure to be finite Borel regular, and in this case we are looking at the measure $\tilde{\mu}$ defined as follows: for $A \subset \bR^n$, $ \tilde{\mu}(A)=\mathbb{P}(\{ \omega \Big| \
there \; exist \;  n, U_0$ $such\; that \; X_n(\omega) \in A    
 \}) $. $\tilde{\mu}$ is clearly a probability measure induced on the Polish space $\bR^n$, and so it is tight by the inverse implication of the Prokhorov's theorem (or we can use the fact that all finite Borel measures defined on a complete metric space is tight). And so it is inner regular; since now $\tilde{\mu}$ is also clearly locally finite, it also implies the outer regularity. 
\end{remark}

Finally, we can use Lemma \ref{onestep} repeatedly to show the convergence result: 
\begin{theorem}\label{theo:convergence}
	By taking $\mu^N_0 = \mu^0$, there exist $\lbrace M_n | M_n \in \bR, n=0,1,..N-1 \rbrace$, $\lbrace L_n | L_n \in \bR, n=0,1,..N-1 \rbrace$  and $\lbrace \delta_n | \delta_n \in \bR, n=0,1,..N-1 \rbrace$ such that 
	\begin{align}
	d(\mu^{N,\cdot}_{N}, \mu_{N}) & \leq \mathlarger{\sum}^{N-1}_{i=0} (2 \kappa^{-2})^i \prod^{i-1}_{j=0} C_{N-j} \bigg( \frac{C}{M_{N-i}} + 2 \delta_{N-i}+ \frac{CM_{N-i}}{\sqrt{K_{N-i}}} +\frac{3}{\sqrt{\caln}} \bigg)
		\end{align}
		where $C_j:=1+C \Delta t L_j M_j$.
		
Then, for any $M > 0$, we have  by picking $\lbrace M_n \rbrace$, $\lbrace K_n \rbrace$, $\caln$ large enough and $\lbrace \delta_n \rbrace$ small enough, then the following hold
\begin{align}
	d(\mu^{N,\cdot}_{N}, \mu_{N}) \leq \frac{C}{M}
\end{align}
for some fixed constant $C$ which depend only on $\kappa$. 
\end{theorem}

\begin{proof}
	With $C_n$ defined as 
\begin{align}
	C_n:=1+C \Delta t L_n M_n
\end{align}
And by using \eqref{recursive_start} repeatedly, we obtain the following result: 
\begin{align}
	d(\mu^{N,\cdot}_{N}, \mu_{N}) & \leq \mathlarger{\sum}^{N-1}_{i=0} (2 \kappa^{-2})^i \prod^{i-1}_{j=0} C_{N-j} (\frac{C}{M_{N-i}}+\frac{CM_{N-i}}{K_{N-i}}+2 \delta_{N-i}+\frac{3}{\sqrt{\caln}})+(2 \kappa^{-2})^{N} \prod^{N-1}_{j=0} C_{N-j} d(\mu^N_0, \mu_0) \\ 
	&\leq \mathlarger{\sum}^{N-1}_{i=0} (2 \kappa^{-2})^i \prod^{i-1}_{j=0} C_{N-j} (\frac{C}{M_{N-i}}+\frac{CM_{N-i}}{K_{N-i}}+2 \delta_{N-i}+\frac{3}{\sqrt{\caln}}) \label{show_zero}
\end{align}
Since we know that $d(\mu^N_0, \mu_0)=0$. 
Now, we just need to show that \eqref{show_zero} vanishes when $K_{l}, N$ gets large and $\delta_{i}$ gets small, $i \in \lbrace 0,1,...,N \rbrace$. 
Notice that $M_l$ comes from the domain truncation for each time step and $\delta_l$ comes from the uniform approximation which are free to choose. The choice of $\delta_l$ will potentially determine the value of $L_n$. 

We fix $M_N:=NM, \delta_N:=\frac{1}{NM}$ where $N$ is the number of time discretization and $M$ is potentially a large number. 

Then, we define $\delta_l$, $M_{l}$ through the following:
\begin{align}
(2 \kappa^{-2})^{i+1} \prod^{i}_{j=0} C_{N-j}2 \delta_{N-i-1} &=(2 \kappa^{-2})^{i} \prod^{i-1}_{j=0} 2 \delta_{N-i}\label{rep1}  \\ 
(2 \kappa^{-2})^{i+1} \prod^{i}_{j=0} C_{N-j} \frac{C}{M_{N-i-1}}&=(2 \kappa^{-2})^{i} \prod^{i-1}_{j=0} C_{N-j} \frac{C}{M_{N-i}} \label{rep2}
\end{align}
Here we define $C_{N+1} \equiv 1$.

Notice that should iterate \eqref{rep1} and \eqref{rep2} iteratively, since defining $\delta_i$ will lead to the lipschitz constant $L_i$ at stage $i$, which is needed for the definition for $C_i$.

Then we have that 
\begin{align}
	\mathlarger{\sum}^{N-1}_{i=0} (2 \kappa^{-2})^i \prod^{i-1}_{j=0} C_{N-j} \frac{C}{M_{n-i}} &\leq N \frac{C}{NM} \nonumber \\
	& \leq  \frac{C}{M}
\end{align}
And we also have 
\begin{align} \mathlarger{\sum}^{N-1}_{i=0} (2 \kappa^{-2})^i \prod^{i-1}_{j=0} C_{N-j} 2 \delta_{n-i} &\leq N \frac{1}{N M} \nonumber \\ 
	& \leq \frac{1}{M}
\end{align}
By picking $K_{N-i}$ to be large, we then can have 
\begin{align}
	\mathlarger{\sum}^{N-1}_{i=0} (2 \kappa^{-2})^i \prod^{i-1}_{j=0} C_{N-j} \frac{CM_{N-i}}{K_{N-i}} &\leq N \frac{1}{N M} \nonumber \\ 
	& \leq \frac{1}{M}
\end{align}
Last but not least, by taking $\caln$ so large such that 
\begin{align}
	(\mathlarger{\sum}^{N-i}_{i=0} (2 \kappa^{-2})^i \prod^{i-1}_{j=0} C_{N-j}) \frac{3}{\sqrt{\caln}} \leq \frac{1}{M}
\end{align}
we can see that \eqref{show_zero} converges to 0 by taking $M$ to be very large. 
\end{proof}

\begin{remark}
Notice that in Theorem \ref{theo:convergence}, it is natural to have terms that depend on $\frac{1}{K_n}$ and $\frac{1}{\caln}$. The presence of $M_n$ and $\delta_n$ are due to technical difficulties. $M_n$ basically gives the growth of the particles in the worst case scenario (we want our domain to be compact), while $L_n$ and $\delta_n$ comes from the Lipchitz approximation for the test function $f$.
\end{remark}

\section{Numerical example}
In this section, we carry out two numerical examples. In the first example, we consider a classic linear quadratic optimal control problem, in which the optimal control can be derived analytically. We use this example as a benchmark example to show the baseline performance and the convergence trend of our algorithm. In the second example, we solve a more practical Dubins vehicle maneuvering problem, and we design control actions based on bearing angels to let the target vehicle follow a pre-designed path.

\subsection{Example 1. Linear quadratic control problem with nonlinear observations}

Assume  B,  K are symmetric, positive definite.
The forward process $Y$ and the observation process $M$ is given by
\begin{equation}
    \label{eq:turestate} 
    \begin{aligned}
        dY(t)&= A(u(t)-r(t))dt +\sigma B u(t)dW_t\\
        dM(t) &= \sin(Y(t))dt + dB_t
    \end{aligned}
\end{equation}
The cost functional is given by
\begin{equation}
    \label{eq:cost} 
    J[u]=\frac{1}{2} E \left[ \int_0^T \langle R(Y_t-Y^{\ast}_t),(Y_t-Y^{\ast}_t) \rangle dt    + \frac{1}{2}\int_0^T\langle Ku_t,u_t \rangle dt  + \frac{1}{2} \langle QY_T,Y_T \rangle \right]
\end{equation}
we want to find $J(u^{\ast}) = \inf_{u \in \mathcal{U}_{ad} [0,T]} J(u)$.

\subsubsection{Experimental design} 
An interesting fact of such example is that one can construct a time deterministic exact solution which depend only on $x_0$.

By simplifying (\ref{eq:cost}), we have
\begin{equation}
    \label{eq:J1} 
    J[u]=\frac{1}{2} \int_0^T( RE [Y_t^T Y_t] - 2RY^{\ast T}_t E[Y_t] + RY^{\ast T}Y^{\ast} +\langle Ku_t,u_t \rangle dt) + \frac{1}{2}E[\langle QY_T,Y_T \rangle]
\end{equation}
Then, we define:
\begin{equation}
\begin{aligned}
\label{eq:J2} 
    X_t &= E[Y_t]=E[Y_0 + \int_0^T A(u(s)-r(s)) ds + \int_0^T \sigma B u(s)dW_s]\\
    &=E[Y_0 + \int_0^T A(u(s)-r(s) ds]
\end{aligned}
\end{equation}

Hence, we see that
\begin{equation}
\begin{aligned}
\label{eq:J3} 
    E [Y_t^T Y_t] &= E[(Y_0 + \int_0^t A(u(s)-r(s)) ds + \int_0^t \sigma B u(s)dW_s)^2]\\
    &=E[Y_0^T Y_0 + \int_0^t (u(s)-r(s))^TA^TA(u(s)-r(s) ds + Y_0^T\int_0^t A(u(s)-r(s)) ds \\
    &+ \int_0^t (u(s)-r(s))^TA^T Y_0 ds ] + E[\int_0^t \sigma^2u(s)^TB^TBu(s)ds]\\
    &=X_t^TX_t + \sigma^2\int_0^t u(s)^TB^TBu(s)ds
\end{aligned}
\end{equation}
And (\ref{eq:J3}) is true because all the terms are deterministic in time given $x_0$. Also, we observe that
\begin{equation}
\begin{aligned}
\label{eq:J4} 
    E [Y_T^T Y_T] &= E[(Y_0 + \int_0^T A(u(s)-r(s)) ds + \int_0^T \sigma B u(s)dW_s)^2]\\
    &=X_T^TX_T + \sigma^2\int_0^T u(s)^TB^TBu(s)ds
\end{aligned}
\end{equation}

As a result, we see that now (\ref{eq:J1}) takes the form:
\begin{equation}
\begin{aligned}
\label{eq:J5} 
    J[u]=&\frac{1}{2} \int_0^T(X_s^TRX_s - 2X_s^TRX_s^{\ast} +X_s^{\ast T}RX_s^{\ast} + u_s^T(\sigma^2 B^TQB+K)u_s)ds \\&+\frac{1}{2} \sigma^2 \int_0^T \int_0^t u_s^TB^TRBu_s dsdt + \frac{1}{2} X_T^TQX_T
    \end{aligned}
\end{equation}
by doing a simple integration by part, we have
\begin{equation}
\begin{aligned}
\label{eq:J6} 
    J[u]=&\frac{1}{2} \int_0^T(X_s^TRX_s - 2X_s^TRX_s^{\ast} +X_s^{\ast T}RX_s^{\ast} + u_s^T(\sigma^2 B^TQB+K)u_s)ds \\&+\frac{1}{2} \sigma^2 \int_0^T (T-t) u_s^TB^TRBu_s ds + \frac{1}{2} X_T^TQX_T
    \end{aligned}
\end{equation}
As a result, we have the following standard deterministic control problem:
\begin{equation}
\begin{aligned}
\label{eq:J7} 
    J[u]=&\frac{1}{2} \int_0^T \underbrace{(X_s^TRX_s - 2X_s^TRX_s^{\ast} +X_s^{\ast T}RX_s^{\ast} + u_s^T(\sigma^2 B^TQB+K)u_s + \sigma^2 (T-t) u_s^TB^TRBu_s)}_{2f} ds \\&+ \frac{1}{2} X_T^TQX_T
    \end{aligned}
\end{equation}
\begin{equation}
\label{eq:J8} 
\frac{dX_t}{dt}=  \underbrace{A(u(t)-r(t))}_{b}, X_{t_0} = X_0
\end{equation}
Then, one can form the following Hamiltonian
\begin{equation}
\label{eq:J9} 
H(x,p,u)=  bp + f
\end{equation}
where $p$ is $\nabla_x v$, $v$ is the value function.

Then, to find the optimal control, we have
\begin{align}
\label{eq:J10}
\frac{\partial}{\partial u} H = 0
\end{align}
which is 
\begin{align}
\label{eq:J11}
Ap + u(\sigma^2B^TRB(T-t)+(K+\sigma^2B^TQB)) = 0
\end{align}
So, we got
\begin{align}
\label{eq:J12}
u_t = \frac{-Ap(t) }{\sigma^2B^TRB(T-t)+(K+\sigma^2B^TQB)}
\end{align}
Also, notice that
\begin{align}
\label{eq:J13}
\frac{d }{dt}p(t) = -R(X_t-X_t^{\ast}), \;\;\;\; p(T) = QX_T
\end{align}
then,
\begin{align}
\label{eq:J14}
p(t) = QX_T + \int_0^T R(X_s-X_s^{\ast})ds
\end{align}
Combine (\ref{eq:J8}),(\ref{eq:J12}),(\ref{eq:J14}) together, we can solve the control of the system.

Then, set
\begin{align}
\label{eq:J15}
    \mathbf A = \begin{bmatrix} 1 & 0.2 & 0.2 & 0.2 \\  0.2 & 1 & 0.2 & 0.2 \\  0.2 & 0.2 & 1 & 0.2 \\  0.2 & 0.2 & 0.2 & 1 \end{bmatrix}
\end{align}
Set B, R, K ,Q be identity matrices. With $X_0=0$, we have the following solution according to this setup. 

To solve (\ref{eq:J14}), let 
\begin{equation*}
\label{eq:J17}
    \begin{aligned}
        X_t^1-X_t^{\ast 1} &:= t \\
        X_t^2-X_t^{\ast 2} &:= cos(t) \\
        X_t^3-X_t^{\ast 3} &:= t^{2} \\
        X_t^4-X_t^{\ast 4} &:= -2\pi sin(2\pi t) 
    \end{aligned}
\end{equation*}
Then we have
\begin{align}
\label{eq:J18}
    p(t) = \begin{bmatrix} X_t^1 \\X_t^2 \\X_t^3 \\ X_t^4\end{bmatrix} + \begin{bmatrix} \frac{T^2}{2}-\frac{t^2}{2} \\ sin(T)-sin(t) \\  \frac{T^3}{3}-\frac{t^3}{3}\\  cos(2\pi T)-cos(2\pi t) \end{bmatrix}
\end{align}

Let $\hat{r}(t)$ as 
\begin{equation*} \label{eq:J19}
\hat{r}_t^1= \frac{-t^2/2}{\beta_t}, \quad \hat{r}_t^2= \frac{-sin(t)}{\beta_t} , \quad \hat{r}_t^3= \frac{-t^3}{3\beta_t}, \quad \hat{r}_t^4= \frac{-cos(2\pi t)}{\beta_t}
\end{equation*}
where $ \beta_t = (1+\sigma^2)+\sigma^2 (T-t) $.
then 
$$r(t)=A\hat{r}(t)$$

Then, plug (\ref{eq:J18}) into (\ref{eq:J12}), then we solve (\ref{eq:J8})
\begin{align}
\label{eq:J20}
\frac{dX_t}{dt} &=  A(u(t)-r(t)) =A(-\frac{Ap(t)}{\beta_t}- A\hat{r}(t)))
\end{align}
we get
\begin{align}
\label{eq:J21}
X_t = \begin{bmatrix} X_t^1 \\X_t^2 \\X_t^3 \\ X_t^4\end{bmatrix} = \frac{\alpha_t A^2}{\sigma^2 }(\begin{bmatrix} \frac{T^2}{2} \\sin(T) \\\frac{T^3}{3} \\ cos(2\pi T)  \end{bmatrix}- \begin{bmatrix} X_T^1 \\X_T^2 \\X_T^3 \\ X_T^4\end{bmatrix})
\end{align}
So, replace $X_t$ with $X_t^{\ast}$, we get

\begin{align}
X_t^{\ast} =  \begin{bmatrix} X_t^{1,\ast} \\X_t^{2,\ast} \\X_t^{3,\ast} \\ X_t^{4,\ast}  \end{bmatrix}  =\begin{bmatrix} t \\cos(t) \\t^2 \\
-2 \pi sin(2\pi t)\end{bmatrix} + \frac{\alpha_t A^2}{\sigma^2 }(\begin{bmatrix} \frac{T^2}{2} \\sin(T) \\\frac{T^3}{3} \\ cos(2\pi T)  \end{bmatrix} - \begin{bmatrix} X_T^1 \\X_T^2 \\X_T^3 \\ X_T^4\end{bmatrix})
\end{align}
where $X_T^i$ can be obtained from the system (\ref{eq:J21}) by letting $t=T$, and 
$$\alpha_t = ln\frac{1+\sigma^2+\sigma^2 T}{(1+\sigma^2)+\sigma^2(T-t)}$$

Then, to find an the exact form by following the trajectory of $y_t$ in this setup, one will have to solve the following Coupled forward-backward ODE.

\begin{align}
	\frac{d X_t}{dt }&=  A(u_t-r_t) \ \ \ \ \ x_{n}=y_{t_n}\\ 
	\frac{d p(t)}{dt }&=X_t-X^*_t, \ \ \ p_T=x^{t_n, y_{t_n}}_T
\end{align}
with $u_t =-A p_t \big/\big( \sigma^2(T-t)+(1+\sigma^2) \big)$. 
As a result, we have
\begin{align}
	\frac{d X_t}{dt }&=  A(-Ap_t \big/\big( \sigma^2(T-t)+(1+\sigma^2) \big)-r_t) \ \ \ \ \ x_{n}=y_{t_n}\\ 
	\frac{d p(t)}{dt }&=X_t-X^*_t, \ \ \ p_T=x^{t_n, y_{t_n}}_T
\end{align}
That is, we need to solve the above coupled FBODE. Then, seeing that $p_t=x^{t_n, y_{t_n}}_T + \int^T_{t_n} X_s-X^*_s ds$. Writing $a_t:=1/\big( \sigma^2(T-t)+(1+\sigma^2) \big)$, we have 
\begin{align}\label{fbode}
	\frac{d X_t}{dt }&=A\big(-a_t A X_T-a_t A\int^T_{t_n} (X_s-X^*_s) ds -r_t \big),  \ \ \ \ \ x_{t_n}=y_{t_n}
\end{align}
To solve \eqref{fbode} numerically, we do a numerical discretization: 
\begin{align}
	x_{t_{n+1}}-x_{t_n}&=-a_{t_n}A^2X_T \Delta t-a_{t_n} (\Delta t)^2 A^2 \sum^{N-1}_{i=n}(X_{t_i}-X^*_{t_i}) - A r_t \nonumber \\ 
	 \Rightarrow  A r_t -a_{t_n} (\Delta t)^2 A^2\sum^{N-1}_{i=n} X^*_{t_i}&=X_{t_n}-X_{t_{n+1}}-a_{t_n} (\Delta t)^2 A^2 \sum^{N-1}_{i=n}X_{t_i}-a_{t_n}A^2X_T, \ \ \ x_{t_n}=y_{t_n} \label{fbode-num}
\end{align}
We can put \eqref{fbode-num} into a large linear system, and solve it numerically.

\subsubsection{Performance experiment}
We set the total number of discretization to be $N=50$.  Set iteration $L=10^4, \sigma=0.1$, number of particles in each dimension is 128, $T=1$, $X_0 = 0$.

In Figure \ref{fig1}, we present the estimated data driven control and the true optimal control, 
\begin{figure}[H]
    \centering
    \includegraphics[width=15cm]{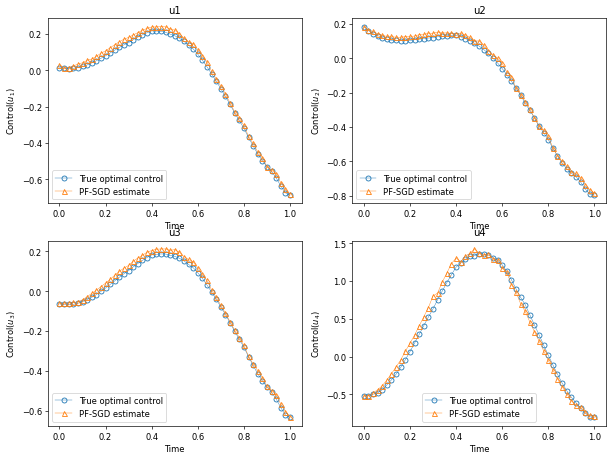} 
        \caption{Estimated control vs True optimal control}
        \label{fig1}
\end{figure}
and in Figure \ref{fig4dpath}, we show the estimated state trajectories with respect to true state trajectories in each dimension.
\begin{figure}[H]
    \centering
    \includegraphics[width=15cm]{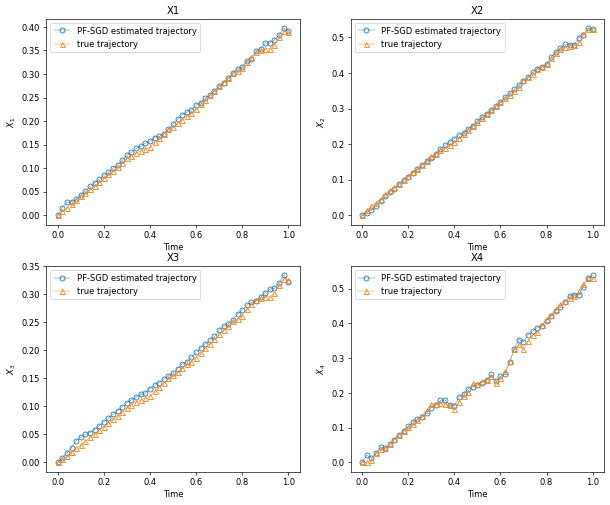} 
        \caption{Estimated state vs True state}
        \label{fig4dpath}
\end{figure}
We can see from those Figures that our data driven feedback control algorithm works very well for this 4-D linear quadratic control problem despite of nonlinear observations.

\subsubsection{Convergence experiment}
In this experiment, we demonstrate the convergence performance of our algorithm, and we study the error decay of the algorithm in the $L_2$ norm with respect to the number of particles used. Each result is an average of $||u^{est}-u^{\ast}||_2$ of 50 independent tests.

Specifically, we set $L=10^4$ and we just increase the number of particles  $S = \{2,8,32,128,512,2048,4096,8192,$ $16384,32768\}$, we got the result in (figure \ref{fig3})
\begin{figure}[H]
    \centering
    \includegraphics[width=8cm]{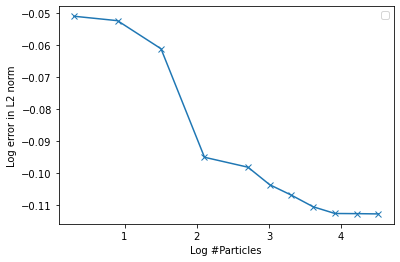} 
        \caption{Error VS Number of Particles}
        \label{fig3}
\end{figure}

Set number of particles $S = \{2,8,32,128,512,1024,2048,4096\}$, $L = S^2$. We got the result in (figure \ref{fig2}).

\begin{figure}[H]
    \centering
    \includegraphics[width=8cm]{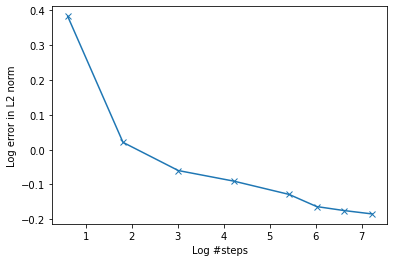} 
        \caption{Error VS Number of Steps}
        \label{fig2}
\end{figure}

From the results above, we can see that the error will decrease and converge as we increase the number of particles and the number of iterations.

\subsection{Example 2. 2D Dubins vehicle maneuvering problem}

In this example, we solve a Dubins vehicle maneuvering problem. The controlled
process is described by the following nonlinear controlled dynamics:
\begin{align}
	dS_t &=\begin{bmatrix} dX_t \\dY_t \end{bmatrix} = \begin{bmatrix} \sin(\theta_t)dt \\ \cos(\theta_t)dt \end{bmatrix} + \sigma dW_t \\  d \theta_t &= u_t dt + \sigma^2 dW_t \\
 dM_t &= [\arctan(\frac{X_t+1}{Y_t-1}), \arctan(\frac{X_t-2}{Y_t-1})]^T + \eta_t
	  \label{dubinscar}
\end{align}
where the pair $(X, Y)$ gives the position of a car-like robot moving in the 2D plane, $\theta$ is
the steering angle that controls the moving direction of the robot, which is governed by the
control action $u_t$, $\sigma$ is the noise that perturbs the motion and control actions. Assume that we do not have direct observations on the robot. Instead, we use two detectors
located on different observation platforms at $(-1, 1)$ and $(2, 1)$ to collect bearing angles of
the target robot as indirect observations. So, we have the observation process $M_t$ Given the expected path $S^{\ast}$, the car should follow it and arrive the terminal position on time. The performance cost functional based on observational
data that we aim to minimize is defined as:
\begin{equation}
    \label{eq:cost2d} 
    J[u]= E \left[ \frac{1}{2}\int_0^T \langle R(S_t-S^{\ast}_t),(S_t-S^{\ast}_t) \rangle dt    + \frac{1}{2}\int_0^T\langle Ku_t,u_t \rangle dt  + \langle Q(S_T-S^{\ast}_T),(S_T-S^{\ast}_T) \rangle \right]
\end{equation}

In our numerical experiments, we let the car start from $(X_0,Y_0)=(0,0)$ to $(X_T,Y_T)=(1,1)$. The expected path $S^{\ast}_t$ is $X_t^2+Y_t^2=1$. Other settings are $T=1$, $\Delta t  =0.02$ i.e. $N_T=50$, $\sigma=0.1$, $\eta_t \sim N(0,0.1)$, $L=1000$, $K=1$ and the initial heading direction is $\pi/2 $. To emphasize the importance of following the expected path and arriving at the target location at the terminal time, let $R=Q=20$.
\begin{figure}[H]
    \centering
    \includegraphics[width=9cm]{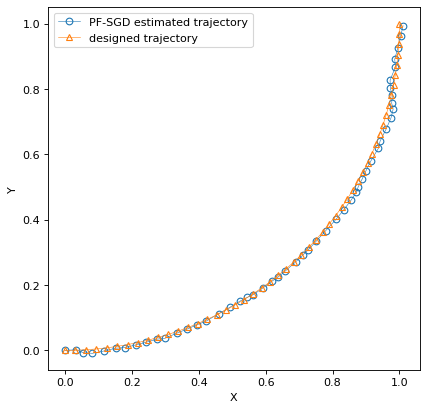} 
    \caption{Controlled trajectory from (0,0) to (1,1)}
    \label{fig5}
\end{figure}

In figure \ref{fig5}, we plot our algorithm's designed trajectory and the estimated trajectory.
We can see from this figure that the car moves
towards the target along the designed path and is “on target” at the final time with a very small error.

If we set $L=10^3$ and we just increase the number of particles $S = \{2,8,32,128,512,1024,2048,$ $4096,8192,16384,32768\}$. To provide the convergence of our algorithm in solving this Dubins vehicle maneuvering problem, we repeat the above experiment 50 times and we got the error = $\frac{1}{M_{rept.}}\sum_{m=1}^{M_{rept.}}\langle (S_t-S^{\ast}_t),(S_t-S^{\ast}_t) \rangle$ in (figure \ref{fig6}) where $M_{rept.}=50$
\begin{figure}[h]
    \centering
    \includegraphics[width=8cm]{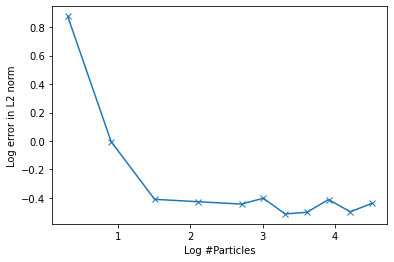} 
        \caption{Error VS Number of Particles}
        \label{fig6}
\end{figure}

Set number of particles $S = \{8,16,32,64,128,256,512,1024\}$, $L = S^2$. We got the error=$\frac{1}{M_{rept.}}\sum_{m=1}^{M_{rept.}}$ 
$\langle (S_t-S^{\ast}_t),(S_t-S^{\ast}_t) \rangle$  in (figure \ref{fig7}), where the error is average of $||S_t-S^{\ast}_t||_2$.

\begin{figure}[H]
    \centering
    \includegraphics[width=8cm]{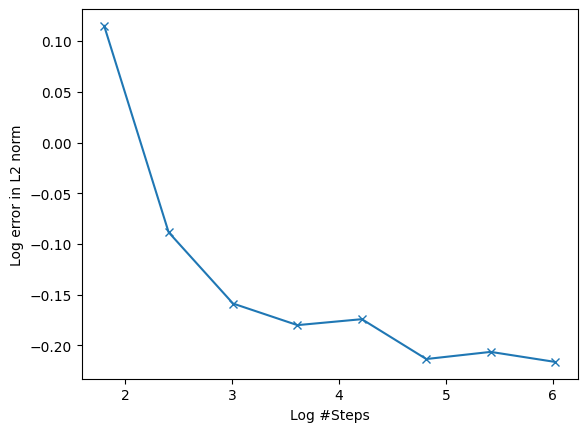} 
        \caption{Error VS Number of Steps}
        \label{fig7}
\end{figure}

From the results above, we can see that the error will decrease and converge as we increase the number of particles and the number of iterations.

\bibliographystyle{plain}

\newpage
\section{Appendix}
\begin{lemma} \label{j_bound}
There exist $C$ such that
	\begin{align}
		\E[||j'^{x}(U_n)||^2] \leq C|x|^2+C
	\end{align}
\end{lemma}
\begin{proof}
	Recall the definition of $j'^{x}(U_n)$, we have 
	\begin{align}
		||j'^x(U_n)||^2 & \leq C \sup_{j \in \lbrace n, ..., N \rbrace} |Y_{j}^x|^2 +C
	\end{align}
	taking expectation on both side and use Lemma 3.3 in \cite{archibald2020STOCHASTIC}, we have that 
	\begin{align}
		\E[||j'^x(U_n)||^2 ] & \leq C \E[\sup_{j \in \lbrace n, ..., N \rbrace} |Y_{j}^x|^2 ] +C \nonumber \\ 
		& \leq C|x|^2+C \label{j_bound1}
	\end{align}
	where \eqref{j_bound} follows again from estimates \textit{Theorem 4.2.1} and \textit{Theorem 5.3.3} in \cite{zhang2017backward}. 

\end{proof}

%
%
%
%
%
%


\end{document}